\documentclass[a4paper,12pt]{article}

\usepackage[backend=biber,style=numeric]{biblatex}

\newif\ifshowkeys
\showkeysfalse

\AtEveryBibitem{%
  \ifshowkeys
    \printtext[brackets]{%
      \fbox{\texttt{\thefield{entrykey}}}%
    }%
    \addspace
  \fi
}

\usepackage{amsthm}{\normalsize }
\usepackage{amsmath}
\usepackage{mathtools,leftindex,tensor}
\usepackage[version=4]{mhchem}
\usepackage{amssymb}
\usepackage{booktabs}
\usepackage{multirow}
\usepackage{bm}
\usepackage{array}
\usepackage{latexsym}
\usepackage{float}
\usepackage{diagbox} 
\usepackage{threeparttable}
\usepackage[textwidth=18cm,textheight=20cm]{geometry}

\usepackage[usenames]{color}
\usepackage[colorlinks=true]{hyperref}
\definecolor{mygray}{gray}{0.9}
\definecolor{deeppink}{RGB}{255,20,147}
\definecolor{mygreen}{rgb}{0.05, 0.576, 0.03}
\definecolor{myred}{rgb}{0.768, 0.09, 0.09}

\usepackage{subfigure}
\usepackage{caption}
\usepackage{indentfirst}
\usepackage{biblatex}
\newtheorem{theorem}{Theorem}[section]
\newtheorem{proposition}[theorem]{Proposition}
\newtheorem{lemma}[theorem]{Lemma}
\newtheorem{corollary}[theorem]{Corollary}
\newtheorem{remark}[theorem]{Remark}

\newtheorem{example}[theorem]{Example}

\newcommand{\R}{\mathbb{R}}

\usepackage{graphicx} 
\title{\bf Ruin theory incorporating MIPP-type jumps}
\author{Dongdong Hu \thanks{Yiwu Industrial \& Commercial College, Yiwu, China. Email: \underline{hudongdong@ywicc.edu.cn}} \and Hasanjan Sayit \thanks{Department of Financial and Actuarial Mathematics, Xi'an Jiaotong Liverpool University, Suzhou, China. Email: \underline{Hasanjan.Sayit@xjtlu.edu.cn}}}

\date{}

\begin{document}

\maketitle

\textbf{Abstract:} The paper studies the ruin probability for an insurer’s surplus process when claim arrivals are governed by a Multiply Iterated Poisson Process (MIPP). This framework generalizes the classical Cramér–Lundberg model by permitting claims to arrive in clusters, which makes it well suited for representing catastrophic insurance events. The main result is the development of explicit conditions for ultimate ruin, demonstrating that the ruin probability is driven chiefly by the jump intensity and the number of iterations. In addition, we derive integro-differential equations for both survival and ruin probabilities and obtain their Laplace transforms. These transforms are connected through recursive relations that tie ruin probabilities at successive iteration levels. We further propose a Cramér–Lundberg-type approximation, yielding asymptotic formulas for ruin probabilities together with a Lundberg-style bound. Finally, numerical experiments—along with comparisons to the classical model—are included to support the theory and to illustrate how claim clustering affects an insurer’s likelihood of ruin.


\section{Introduction}

The Cramér–Lundberg model is the classical starting point of ruin theory. Originally introduced by Filip Lundberg in 1903 and later placed on a rigorous stochastic-process foundation by Harald Cramér in 1930, it has remained for decades the central framework in insurance risk theory; see, for example, \cite{asmussen2010}, \cite{Gerber1979}, and \cite{Grandell1991}. In its standard form, an insurer’s surplus evolves as a deterministic premium inflow minus a compound Poisson stream of independent and identically distributed claim sizes. Despite its analytical elegance and tractability, the model relies heavily on the assumption that claim arrivals are independent and memoryless, an assumption that is often inconsistent with empirical evidence in catastrophe, liability, and other heavy-tailed insurance portfolios. In such lines of business, claims frequently arrive in bursts or clusters rather than as isolated events. This empirical mismatch has motivated a large literature on risk models with richer arrival mechanisms, including renewal processes, Cox processes, Hawkes processes, and various classes of time-changed or subordinated Poisson processes; see \cite{Willmot2017}, \cite{Albrecher2020}, and \cite{Palmowski2025}.

An important extension of the Cramér–Lundberg model is to relax the assumption that claims arrive singly according to a homogeneous Poisson process and to allow the claim-counting mechanism itself to exhibit dependence, clustering, overdispersion, or stochastic intensity \cite{asmussen2010}, \cite{Gerber1979}, \cite{Grandell1991}. In the classical model, interarrival times are independent and exponentially distributed, so the aggregate claim process is a compound Poisson process \cite{asmussen2010}, \cite{rolski1999}. This assumption is analytically convenient but often unrealistic for catastrophe, liability, cyber, operational-risk, and heavy-tailed insurance portfolios, where a single triggering event may generate a cascade of claims, and where arrival rates may vary over time or be self-exciting \cite{Albrecher2020}, \cite{Cheung2018}, \cite{Hu-2026}, \cite{Palmowski2025}, \cite{selch2016}. To accommodate such features, the literature has replaced the Poisson process by a variety of more general mechanisms: renewal processes in the Sparre Andersen framework \cite{Prabhu-1961}, \cite{Thorin1971FurtherRO}, \cite{Willmot2017}; mixed Poisson and Cox processes with stochastic intensity \cite{Albrecher2010}, \cite{selch2016} , \cite{TOMITA-2022}; Hawkes self-exciting point processes \cite{Palmowski2025}; subordinated or time-changed Poisson processes \cite{Klinge2026}, \cite{Orsingher_Toaldo_2015}; and general non-negative Lévy subordinators \cite{Bertoin1994}, \cite{Kyprianou2014}. These extensions preserve the basic economic decomposition of the surplus as premium income minus aggregate claims, but they change the net profit condition, the integro-differential equation for the survival probability, the Laplace transform, the adjustment coefficient, and the asymptotic tail behaviour of the ruin probability \cite{asmussen2010}, \cite{Baltrunas-1999}, \cite{Cheung2018}, \cite{Dufresne-Gerber1989}, \cite{Kyprianou2014}, \cite{Shimizu2021}. Methodologically, they require tools from fluctuation theory for spectrally negative Lévy processes, scale functions, Volterra integral equations, martingale arguments, and Cramér–Lundberg-type asymptotics \cite{Kelbert2012}, \cite{GERBER1973}, \cite{Grahovac2018}, \cite{2016Yuliya}, \cite{Polyanin}, \cite{rolski1999}. 
The MIPP-based risk model examined in this paper fits squarely within this wider research agenda, yet it holds a particular niche: rather than postulating a stochastic intensity or a self-exciting structure, it builds the claim-counting process through repeated composition of independent Poisson processes \cite{Hu-2026}, \cite{Orsingher_Toaldo_2015}. This produces a hierarchical subordinator that can capture multi-stage or cascading claim arrivals while remaining highly amenable to analysis \cite{Constantinescu2018}, \cite{Klinge2026}. Within the present MIPP framework, the general programme is sharpened by the recursive Laplace transform, scale-function duality, and Cramér–Lundberg approximation developed in Sections 5–6, with the standard model recovered when $n=1$. The MIPP setting also has a natural reading in catastrophe risk modelling: a base Poisson process generates catastrophic events, and each event activates a secondary Poisson process producing smaller claims, which may themselves spawn additional layers further down the hierarchy. This multilevel structure reflects the cascading character of real catastrophe losses, where a single large shock—such as a hurricane, an earthquake, or an industrial accident—produces a cluster of insurance claims that emerges over time. By nesting Poisson processes recursively, one obtains a flexible and probabilistically consistent description of clustered claim arrivals.

The MIPP model is intended to complement, not replicate, other generalisations such as Hawkes processes \cite{Palmowski2025}, Cox processes, mixed Poisson processes, and more general subordinated counting processes \cite{Orsingher_Toaldo_2015}. Hawkes and Cox specifications incorporate stochastic intensity and/or self-excitation, often reducing analytical tractability. By contrast, MIPP relies on deterministic iterated Poisson composition, yielding a non-negative Lévy subordinator. As a result, one can derive exact Laplace transforms, recursion formulas, scale functions, and Cramér–Lundberg asymptotics. In this way, the MIPP framework offers a tractable intermediate option between the classical Poisson model and richer point-process approaches to catastrophic clustering.

While the probabilistic properties of the MIPP were investigated in \cite{Hu-2026}, their consequences for insurer solvency have, to date, received little direct study. In the wider ruin-theory literature, compound Poisson risk models with heavy-tailed claim sizes have been analysed in depth; see, e.g., \cite{Constantinescu2018}, \cite{Grahovac2018}, \cite{Ramsay2003}, and \cite{Usabel2001}. More general risk functionals have been effectively addressed using Gerber–Shiu functions and algebraic operator techniques \cite{Albrecher2010}. Lévy-driven surplus processes and the asymptotics of their ruin probabilities are considered in \cite{Bertoin1994}, \cite{Kelbert2012}, and \cite{Kyprianou2014}; results on smoothness properties, bounds, and supermartingale methods can be found in \cite{2016Yuliya} and \cite{rolski1999}. Volterra integral equation formulations of ruin probabilities appear in \cite{Seal1969} and are treated systematically in \cite{Polyanin}. Delay differential equations, arising in some discrete claim-size frameworks, are discussed in \cite{Polyanin2023}. Asymptotic statistical inference for risk processes is developed in \cite{Shimizu2021}. By contrast, the particular setting where the claim-counting mechanism itself has a hierarchical iterated structure has been studied only sparingly. The present work addresses this gap by providing a systematic ruin analysis for a surplus process driven by an MIPP.

In this paper, we analyse a ruin model in which jumps are generated by an MIPP. Specifically, we consider the insurer’s surplus process
\begin{equation}\label{one1}
X_t^{(n)}(x) = x + ct - \sum_{i=0}^{V_t^{(n)}} \xi_i, \; \; t \geq 0,
\end{equation}
where $x>0$ is the initial capital, $c>0$ denotes a constant premium rate, $V_t^{(n)}$ is an MIPP counting process with intensity $\lambda>0$, and $\{\xi_i\}$ are i.i.d. positive claim amounts with finite mean $m = E\xi_i$. At each jump epoch, $V_t^{(n)}$ can increase by more than one, representing the arrival of a claim batch (cluster). The aggregate claim amount for such a batch is the sum of a random number of individual claims, where the batch size is distributed as $V_1^{(n-1)}$, i.e., the value of the inner MIPP process evaluated at time 1.

The ruin probability
\[
\Pi_n(x)=P\big(\inf_{t\geq 0}X_t^{(n)}(x)<0\big)
\]
denotes the likelihood that the insurer’s surplus will eventually fall below zero. It is widely regarded as the central risk measure in insurance risk theory, as it delivers a direct and intuitive indicator of an insurer’s long-term solvency. In contrast to finite-horizon ruin probabilities—which evaluate the chance of insolvency within a fixed time interval and are mainly used for short-term operational purposes—the ultimate ruin probability reflects the insurer’s lifetime insolvency risk and is thus crucial for strategic planning, regulatory supervision, and structuring reinsurance arrangements. 
From a regulatory perspective, $\Pi_n(x)$ underpins the determination of adequate capital buffers: supervisors require insurers to hold enough surplus to ensure the probability of ruin stays below specified thresholds. This concern sits at the heart of frameworks such as Europe’s Solvency II and the United States’ Risk-Based Capital (RBC) system, both of which focus on tail risk and extreme loss outcomes. For internal capital allocation, $\Pi_n(x)$ is instrumental in apportioning economic capital across business lines or geographic areas, promoting coherent risk-based pricing and allowing insurers to refine their portfolios in response to the underlying claim arrival mechanism. Moreover, in reinsurance design and pricing, the ruin probability informs the structure and negotiation of treaties by measuring the cedent’s exposure to large, potentially clustered loss events, helping to set an appropriate level of risk transfer and the associated premium.

In the MIPP setting, the ruin probability $\Pi_n(x)$ captures the layered structure of the claim count process; consequently, it depends not only on the initial surplus $x$, premium rate $c$, and the claim size distribution, but also on the iteration level $n$ and the intensity parameter $\lambda$. This expanded dependence supports a more realistic assessment of solvency risk when claims emerge in cascading, clustered patterns—common in catastrophe insurance, where a single major event can trigger multiple follow-on losses. By formulating explicit ultimate ruin conditions and deriving integro-differential equations, Laplace transforms, and asymptotic approximations for $\Pi_n(x)$, this paper presents a unified analytical approach linking classical ruin theory with the clustered claim arrival processes observed in practice, thereby providing both theoretical contributions and practical tools for solvency assessment, capital management, and reinsurance planning.

The principal contributions of this work are as follows. First, we pinpoint the key premium threshold $\lambda^n m$ in relation to the mean claim intensity. Specifically, when $c\le \lambda^n m$, ruin happens with probability one (i.e., $\Pi_n(x)=1$ for every $x>0$); whereas if $c>\lambda^n m$, ruin is not guaranteed and $\Pi_n(x)\to 0$ as $x\to\infty$. This generalizes the classical net profit condition to the MIPP setting and highlights that the iteration level $n$ effectively scales up the claim intensity.

Second, we establish integro-differential equations governing the survival probability $\Lambda_n(x)=1-\Pi_n(x)$ and show that the right-derivative of $\Lambda_n(x)$, which we denote by $\Lambda_n^+(x)$, satisfies
\[
 c\Lambda_n^{+}(x)=\lambda q_{n-1}\Lambda_n(x)-\lambda\sum_{k=1}^{\infty}P(V_1^{(n-1)}=k)\int_0^x\Lambda_n(x-y)\,dF^{*k}(y),
\]
where $q_{n-1}=P(V_1^{(n-1)}>0)$ and $F^{*k}$ denotes the $k$-fold convolution of the claim-size distribution $F$. Combined with the boundary condition $\Lambda_n(0)=1-\lambda^n m/c$, this relation uniquely characterizes the survival probability as the solution to a second-kind Volterra integral equation.

Finally, we perform a complete Laplace-transform study of the ruin probability and derive the closed-form expression
\[
\mathcal{L}_{\Lambda_n}(\theta)=\frac{c-\lambda^nm}{c\theta-\lambda\bigl(1-g_{n-1}^{\circ}(E[e^{-\theta\xi}])\bigr)}, \quad \theta>0,
\]
with $g_{n-1}^{\circ}$ denoting the $(n-1)$-fold composition of $g(x)=e^{-\lambda+\lambda x}$. This identity provides a workable route to compute ruin probabilities through numerical Laplace inversion and also reveals a recursive pattern:
\[
\mathcal{L}_{\Lambda_{n+1}}(\theta)=\frac{c-\lambda^{n+1}m}{c\theta-\lambda\left(1-\exp\left\{\frac{c-\lambda^nm}{\mathcal{L}_{\Lambda_n}(\theta)}-c\theta\right\}\right)}.
\]
The recursion allows one to obtain the Laplace transform of the survival function at level $n+1$ directly from that at level $n$, enabling an efficient computational procedure.

We also derive a Cramér–Lundberg-type approximation for the ruin probability in the MIPP context. Assuming the net profit condition $c>\lambda^n m$, we show that
\[
\Pi_n(x)\sim C_n e^{\hat{R}_n x}, \quad x\to\infty,
\]
where $\hat{R}_n>0$ denotes the unique positive solution to the Cramér–Lundberg equation
\[
\lambda (g_{n-1}^{\circ}(E e^{R\xi})-1)=cR,
\]
and where $C_n$ is provided in closed form. This result generalizes the classical Cramér–Lundberg asymptotics to the MIPP setting and supplies a practical approximation when the initial reserve is large.
Lastly, we perform a numerical study to illustrate how the ruin probability varies with the iteration index $n$, the rate $\lambda$, and the claim size distribution. We validate the theoretical findings via numerical Laplace inversion and by benchmarking the Cramér–Lundberg approximation against exact ruin probabilities. The computations suggest that for $\lambda>1$, increasing $n$ markedly raises the ruin probability, reflecting more pronounced claim clustering. When $\lambda=1$, the ruin probability approaches $m/c$ as $n\to\infty$, whereas for $\lambda<1$ it tends to zero, highlighting a phase transition in the model’s limiting regime.

The rest of the paper is structured as follows. Section \ref{S:2} collects preliminaries on the MIPP, including the jump-time law, probability mass function, characteristic exponent, and moment properties. Section \ref{S:3} states the principal results on ruin, covering the critical threshold theorem, the limit behavior as $n\rightarrow \infty$, and a generalized net profit condition. Section \ref{S:4} derives the integro-differential equations for survival and ruin probabilities and explains how, for discrete claim sizes, these can be simplified to delay differential equations. Section \ref{S:5} introduces the Laplace-transform approach, giving a recursive formula for the Laplace transform of the survival probability and constructing the associated scale function. Section \ref{S:6} proves the Cramér–Lundberg approximation, supplies explicit formulas for the adjustment coefficient and the asymptotic constant, and establishes the Lundberg inequality. Section \ref{S:7} concludes and sketches possible directions for future work.

\section{Preliminaries}\label{S:2}

This Section provides the probabilistic background on the Multiply Iterated Poisson Process (MIPP) that underpins the ruin analysis. Following \cite{Hu-2026}, we define the MIPP by recursively composing independent Poisson processes, yielding a hierarchical subordinator well suited to clustered claim arrivals; closely related constructions based on stochastic time changes and random jumps are studied in \cite{Orsingher_Toaldo_2015}, while the general Lévy-process framework is treated in \cite{Kyprianou2014}. We then collect the distributional tools needed later: the exponential sojourn-time law, the recursive probability mass function, the characteristic exponent, and the moment behaviour. The recursive characteristic exponent 
 $\ell_n(\theta)=\lambda e^{\ell_{n-1}(\theta)}-\lambda$ is related to continued exponentials \cite{Barrow-1936} and is central to the Laplace-transform analysis in Section \ref{S:5}. We also include graphical illustrations of the PMF and moments for different $\lambda$ and iteration levels $n$, which validate the MIPP properties obtained in \cite{Hu-2026} and anticipate the phase transition in Theorem \ref{th0514}. For classical ruin-theoretic background and related asymptotic results, see \cite{Albrecher2010}, \cite{Baltrunas-1999}, \cite{Bertoin1994}, \cite{2016Yuliya}, and \cite{Polyanin}.

Before turning to a more detailed analysis of MIPPs, we first restate their definition, adopting the same notation as in \cite{Hu-2026}. Let $N_t^1, N_t^2, \ldots, N_t^n$ be $n$ independent Poisson processes, each with rate $\lambda$. Set $U_t^{(1)} = N_t^1, U_t^{(2)} = N_t^2, \ldots, U_t^{(n)} = N_t^n$. For integers $1 \le k \le l \le n$, define
\[
V_t^{(k,k)} = U_t^{(k)}, \quad
V_t^{(k,k+1)} = U^{(k+1)}\bigl(V_t^{(k,k)}\bigr), \quad
V_t^{(k,k+2)} = U^{(k+2)}\bigl(V_t^{(k,k+1)}\bigr), \ \ldots,\ 
V_t^{(k,l)} = U^{(l)}\bigl(V_t^{(k,l-1)}\bigr).
\]
Through this recursive scheme, $V_t^{(k,l)}$ is obtained by iteratively composing the Poisson processes $N_t^k, N_t^{k+1}, \ldots, N_t^l$. We refer to $V_t^{(k,l)}$ as a multiply iterated Poisson process (MIPP, for short). As in \cite{Hu-2026}, and for notational simplicity, we write
\[
V_t^{(l)} := V_t^{(1,l)}, \quad \forall\; l \geq 2.
\]
From the above definition, we immediately have
\begin{equation}\label{kVn}
V_t^{(n)}=N^{n}_{V_t^{(n-1)}}, \; \; V^{(n)}_t={V}^{(2, n)}_{N_t^1}.
\end{equation}
Hence, the process $V_t^{(n)}$ is a specific instance of the class of models considered in \cite{Orsingher_Toaldo_2015}. More precisely, in their notation, $V_t^{(n)}$ corresponds to $N^{f}(t)$, with the subordinator $H^{f}(t)$ identified as $V_t^{(n-1)}$. Observe also that $V_t^{(2, n)}$ and $V_t^{(n-1)}$ have the same distribution.

The MIPP process was analyzed in \cite{Hu-2026} within the framework of ruin theory. In particular, explicit expressions were obtained for the jump times and sojourn times of MIPP, which we summarize here. Let $\tau_k^{(n)}$ denote the $k$-th jump time of $V_t^{(n)}$. It is shown in \cite{Hu-2026} that the sojourn time $S^{(n)}\overset{d}{=}\tau_{k+1}^{(n)}-\tau_{k}^{(n)}$ is exponentially distributed with rate parameter $\lambda P(V^{(n-1)}_1>0)$, i.e.,
\[
S^{(n)}\sim \mathrm{EXP}\big(\lambda P(V^{(n-1)}_1>0)\big).
\]
The probability mass function (PMF) of the MIPP is given by
\begin{equation}\label{71}
\begin{split}
P(V_t^{(n)}=k)=\frac{\lambda^k}{k!}\sum_{j=0}^{+\infty} j^k e^{-\lambda j} P(V_t^{(n-1)}=j), \quad k\geq 0.
\end{split}
\end{equation}

Using this formula, we plot the probability mass function of $V_1^{(n)}$ for various values of $n$ and $\lambda$. The figures show the following trends: for $\lambda=1/2$, the mass increasingly concentrates at the origin as $n$ increases; for $\lambda=1$, the distributions become progressively more spread out with larger $n$; and for $\lambda=2$, at any fixed time the support expands and the distribution grows steadily more diffuse as $n$ rises. These qualitative patterns agree with the results reported in the Appendix of \cite{Hu-2026}.

\begin{figure}[htb]
\centering
\subfigure[$n=1$]{
\includegraphics[width=5cm]{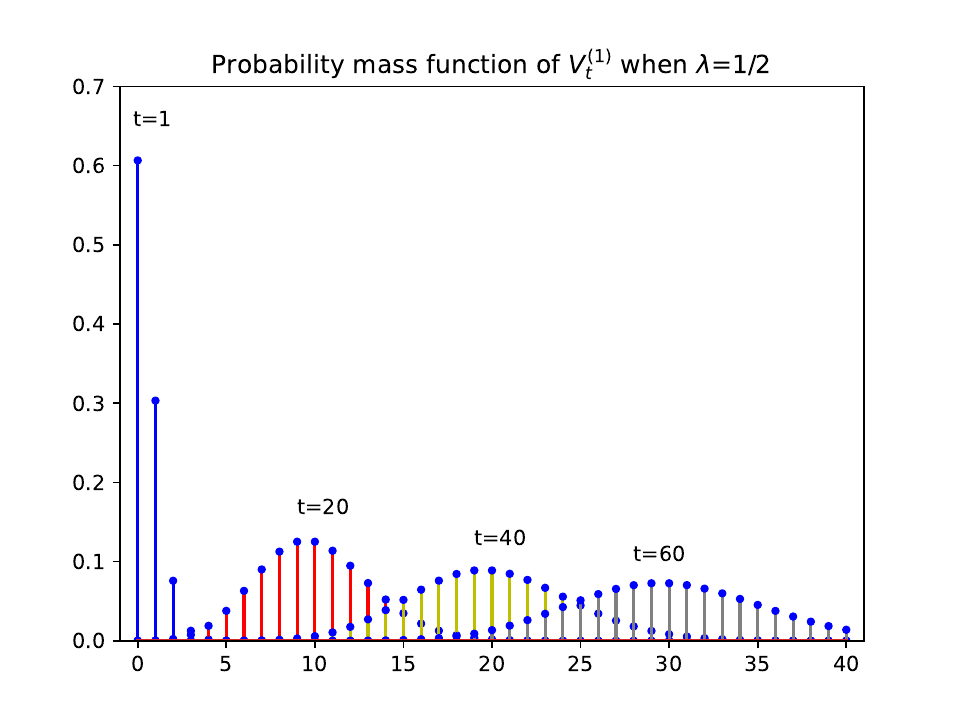}
}
\quad
\subfigure[$n=2$]{
\includegraphics[width=5cm]{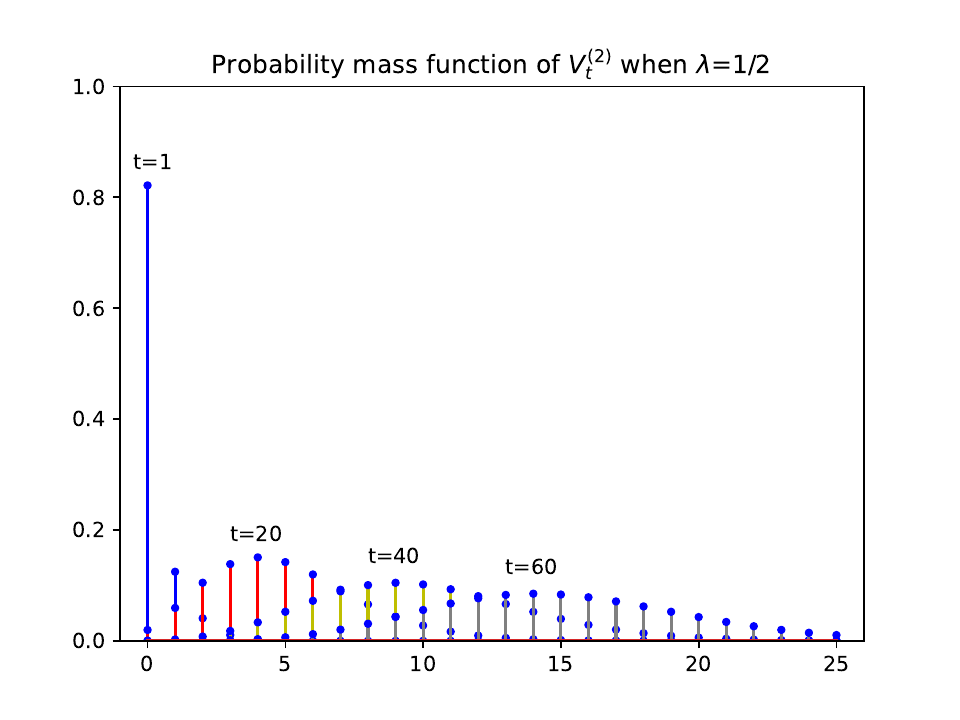}
}
\quad
\subfigure[$n=3$]{
\includegraphics[width=5cm]{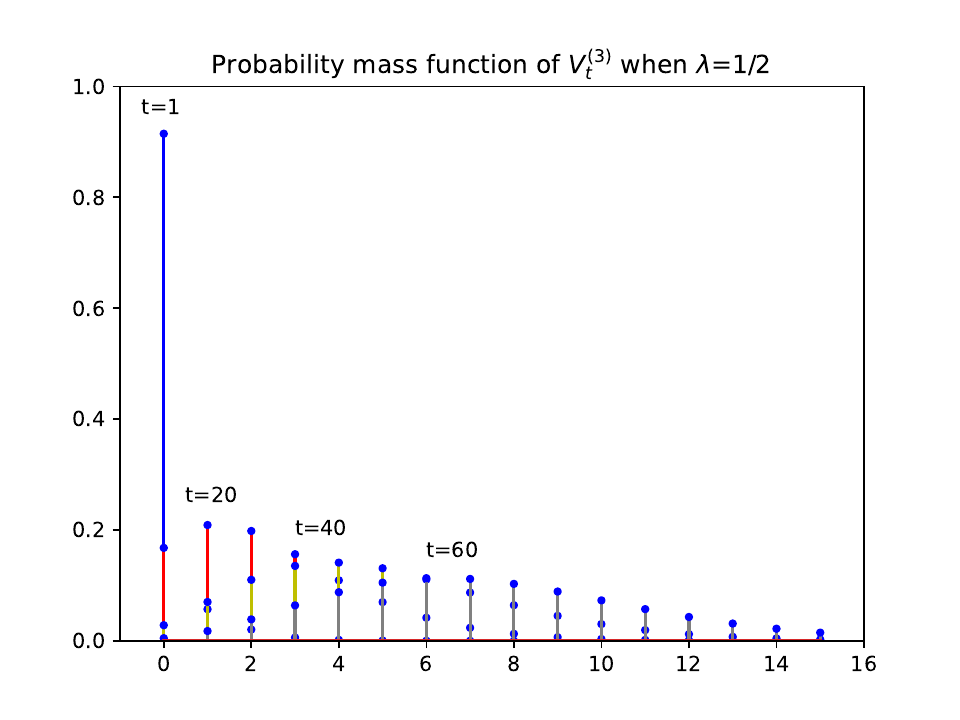}
}
\caption{Fixing $\lambda=\frac{1}{2}$, the PMF for $n=1,2,3$ as functions of different $t$}
\end{figure}

\begin{figure}[htb]
\centering
\subfigure[$n=1$]{
\includegraphics[width=5cm]{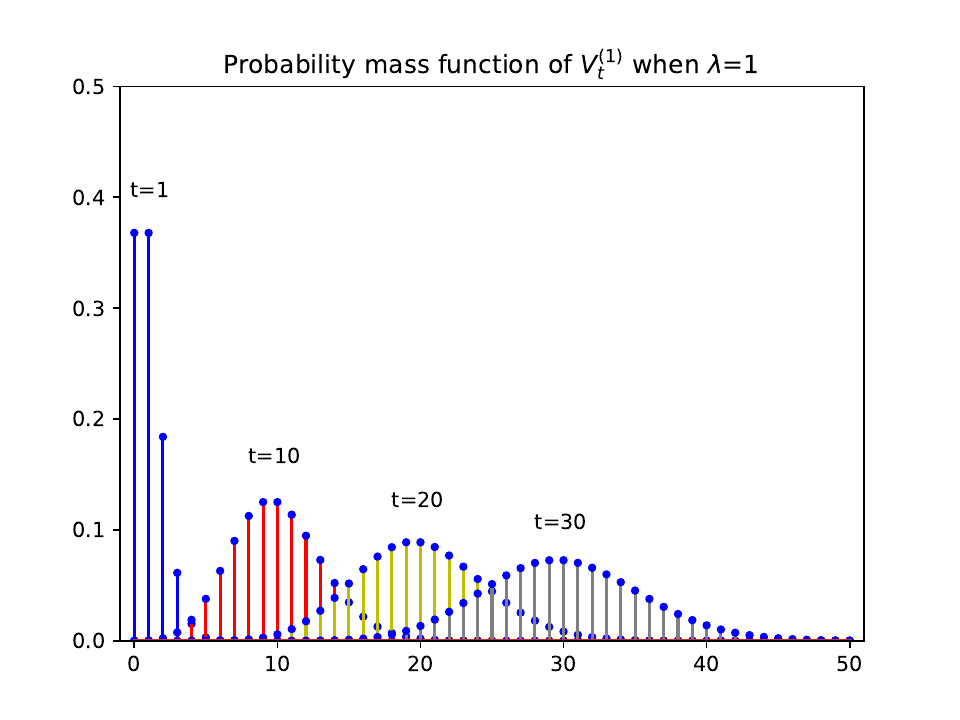}
}
\quad
\subfigure[$n=2$]{
\includegraphics[width=5cm]{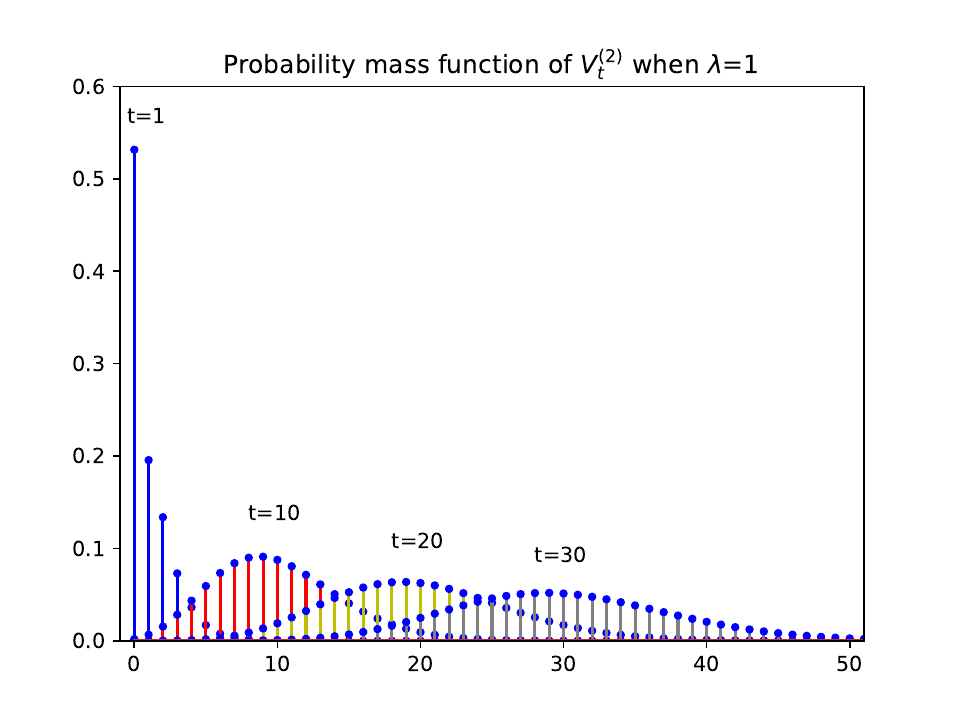}
}
\quad
\subfigure[$n=3$]{
\includegraphics[width=5cm]{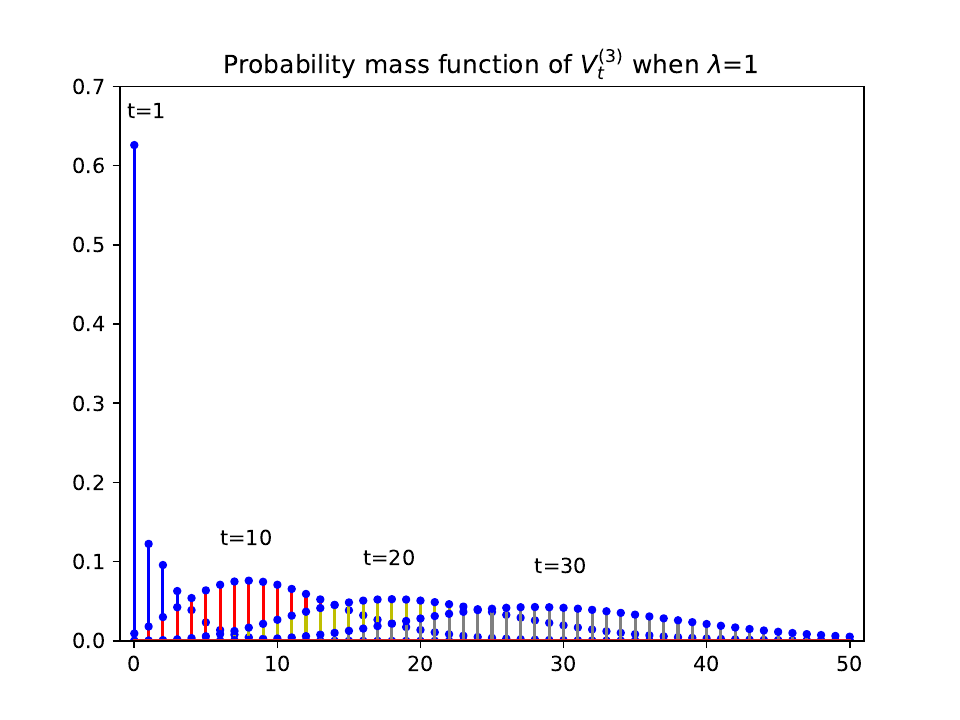}
}
\caption{For $\lambda=1$, the PMF of MIPP for $n=1,2,3$ across different $t$}
\end{figure}

\begin{figure}[htb]
\centering
\subfigure[$n=1$]{
\includegraphics[width=5cm]{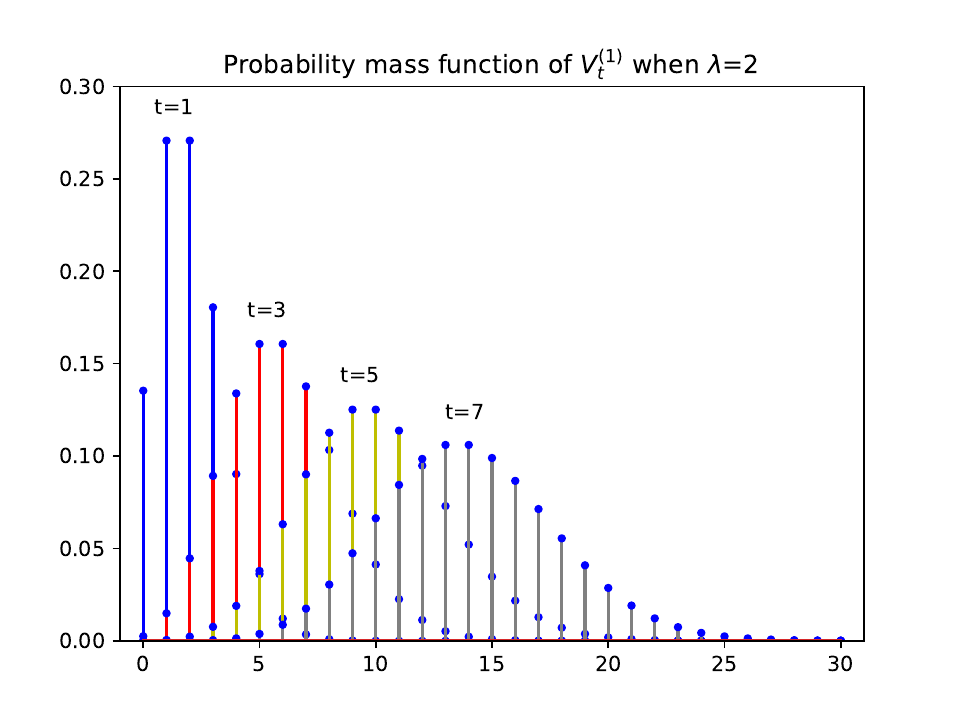}
}
\quad
\subfigure[$n=2$]{
\includegraphics[width=5cm]{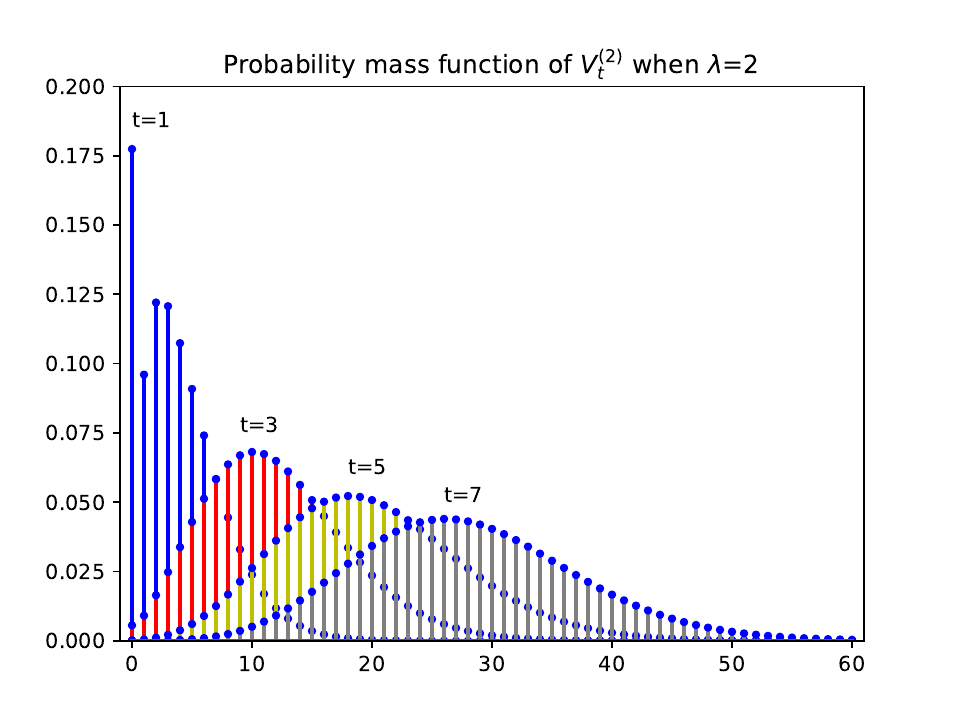}
}
\quad
\subfigure[$n=3$]{
\includegraphics[width=5cm]{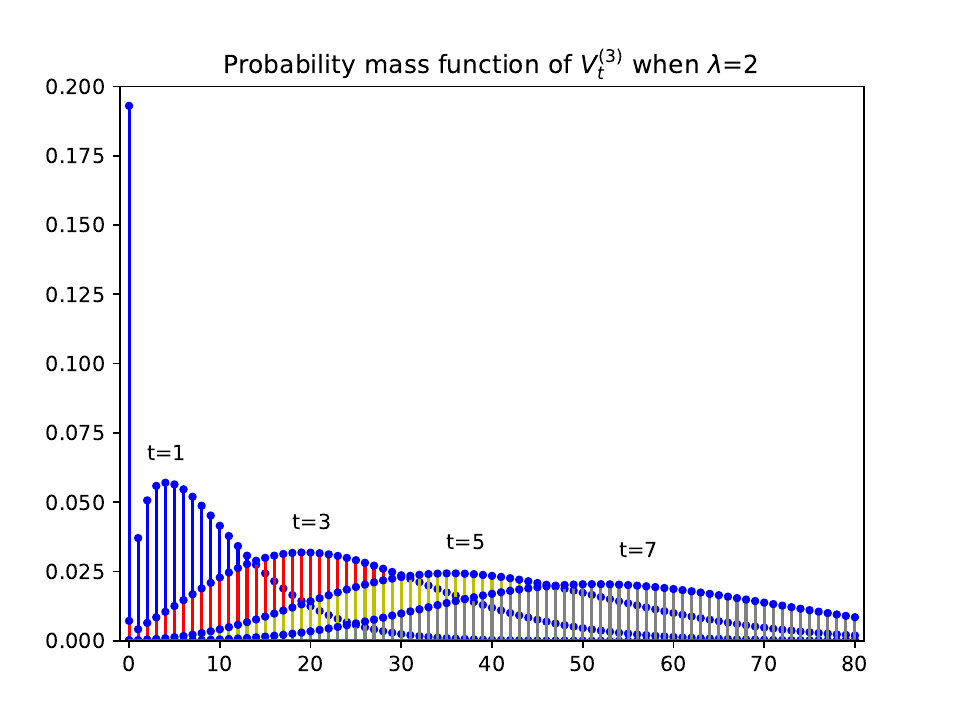}
}
\caption{For $\lambda=2$, the PMF of MIPP for $n=1,2,3$ as functions of different $t$}
\end{figure}

\newpage

It is clear that $V_t^{(n)}$ is a subordinator, i.e., a non-negative L\'evy process whose sample paths are non-decreasing, and it can be illustrated as follows.
\begin{figure}[htb]
\centering
\subfigure[$\lambda=1/2$]{
\includegraphics[width=5cm]{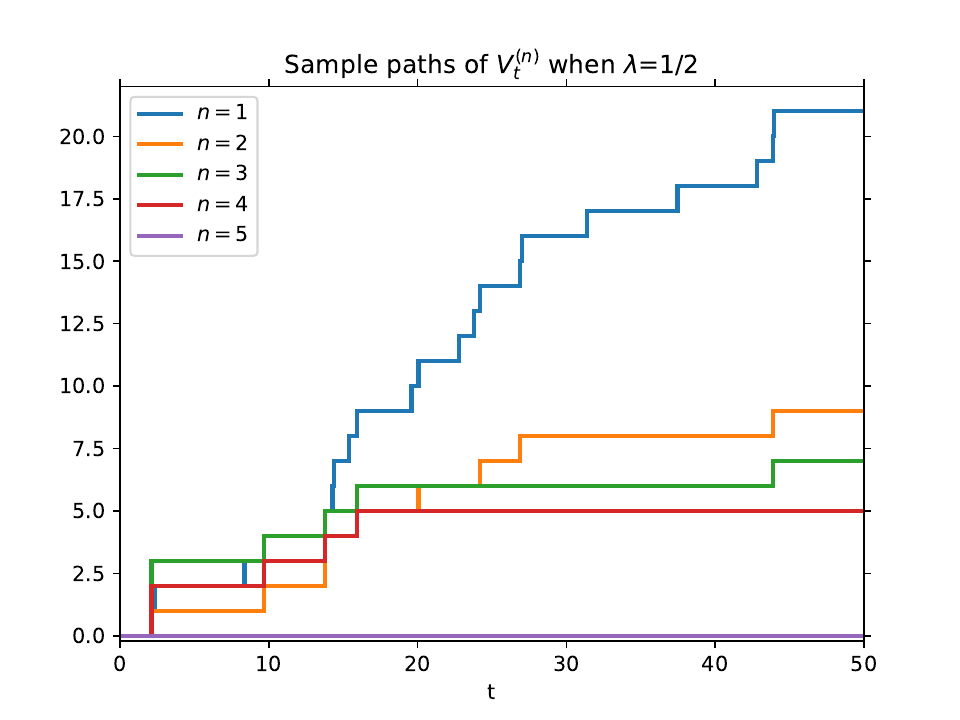}
}
\quad
\subfigure[$\lambda=1$]{
\includegraphics[width=5cm]{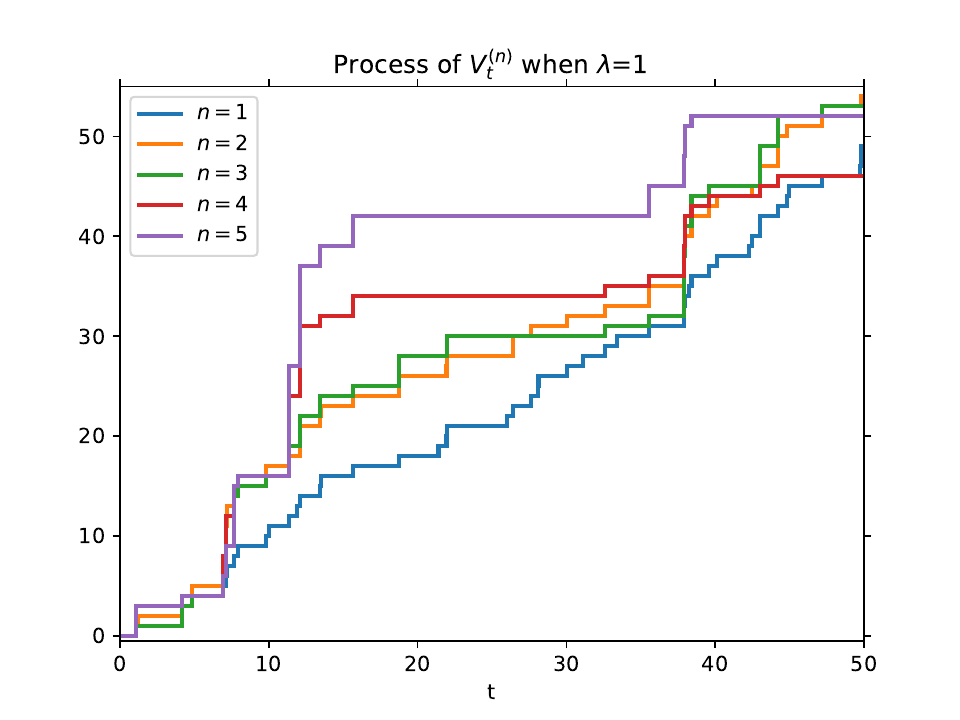}
}
\quad
\subfigure[$\lambda=2$]{
\includegraphics[width=5cm]{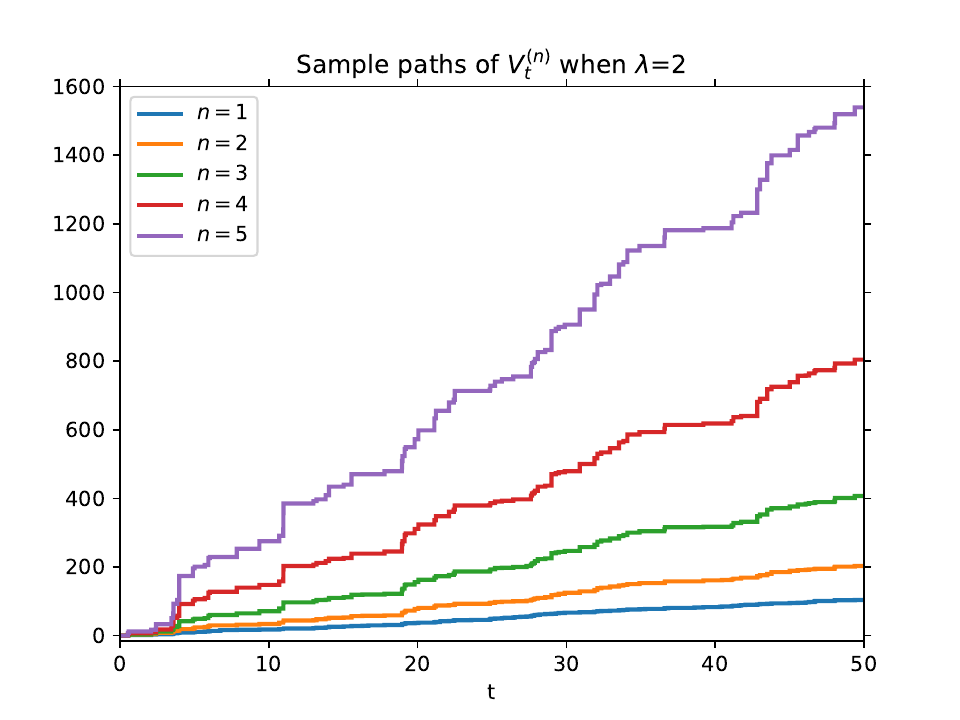}
}
\caption{The sample paths of $V_t^{(n)}$ w.r.t. different $\lambda$'s}
\end{figure}

Let $\ell_1(\theta)=\lambda(e^{\theta}-1)$ denote the characteristic exponent of the Poisson process. For $n\geq 1$, we write $\ell_n(\theta)$ for the characteristic exponents of $V_t^{(n)}$. One can easily check that
\begin{equation}\label{ell}
\ell_n(\theta)=\lambda e^{\ell_{n-1}(\theta)}-\lambda, \; \; n\geq 2.
\end{equation}
Indeed, note that
\begin{equation}
\begin{split}
e^{t\ell_n(\theta)}=&Ee^{\theta V_t^{(n)}}=\sum_{k=0}^{+\infty}Ee^{\theta V_k^{(2,n)}}P(N^1_t=k)=\sum_{k=0}^{+\infty}(Ee^{\theta V_1^{(n-1)}})^kP(N^1_t=k)\\
=&\sum_{k=0}^{+\infty}e^{k\ell_{n-1}(\theta)}P(N^1_t=k)=e^{t(\lambda e^{\ell_{n-1}(\theta)}-\lambda)},
\end{split}
\end{equation}
which immediately yields (\ref{ell}). Applying the formulas in the Appendix of \cite{Hu-2026}, we can evaluate the mean, variance, skewness, and kurtosis at $t=1$ for different choices of $\lambda$. For $\lambda=\frac{1}{2}$, both the mean and variance tend to zero as $n$ increases, while the skewness and kurtosis blow up to infinity. For $\lambda=1$, the mean stays the same; the variance and kurtosis seem to grow proportionally with $n$, and the skewness also rises with $n$ but more slowly than linearly. When $\lambda=2$, the mean and variance grow very rapidly as $n$ increases, whereas the skewness and kurtosis tend to stabilize at constant values as $n$ becomes large.

\begin{figure}[htb]
\centering
\subfigure[Mean]{
\includegraphics[width=3.8cm]{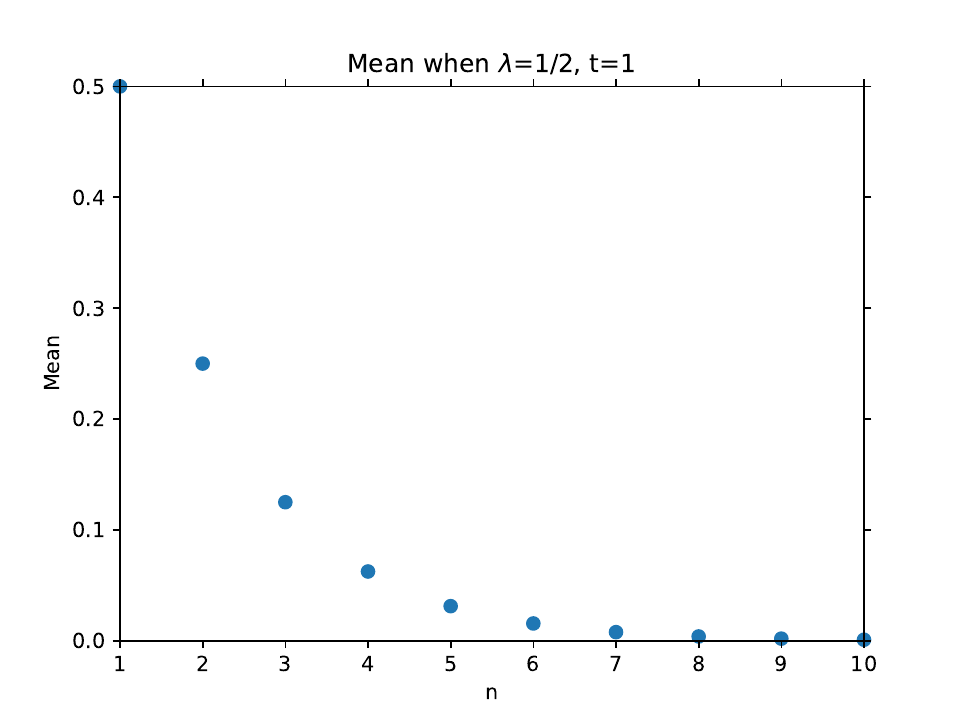}
}
\quad
\subfigure[Variance]{
\includegraphics[width=3.8cm]{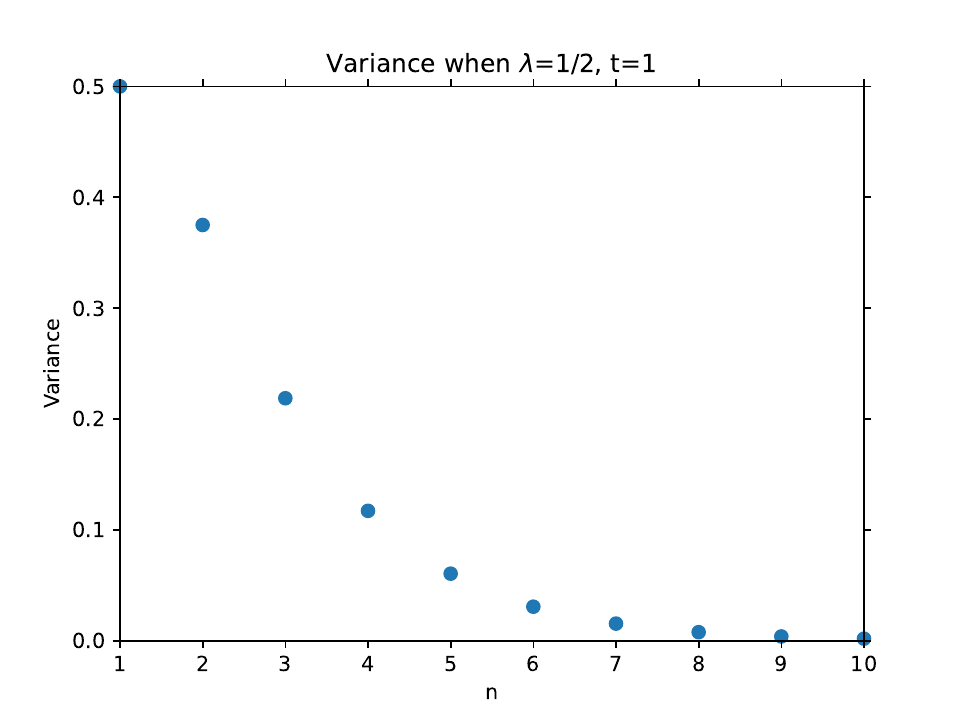}
}
\quad
\subfigure[Skewness]{
\includegraphics[width=3.8cm]{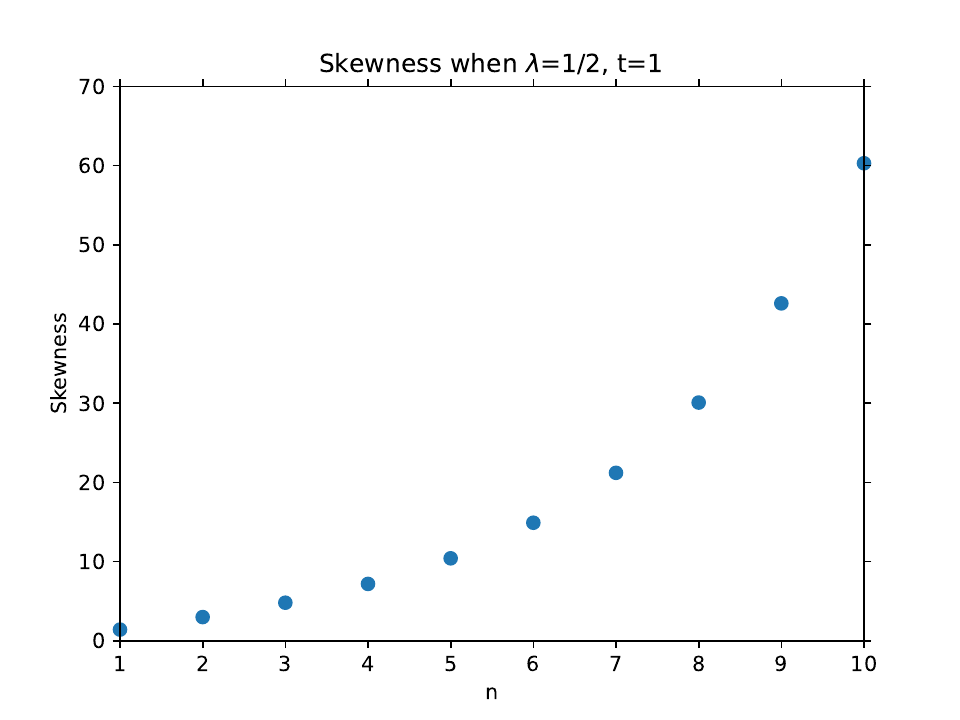}
}
\quad
\subfigure[Kurtosis]{
\includegraphics[width=3.8cm]{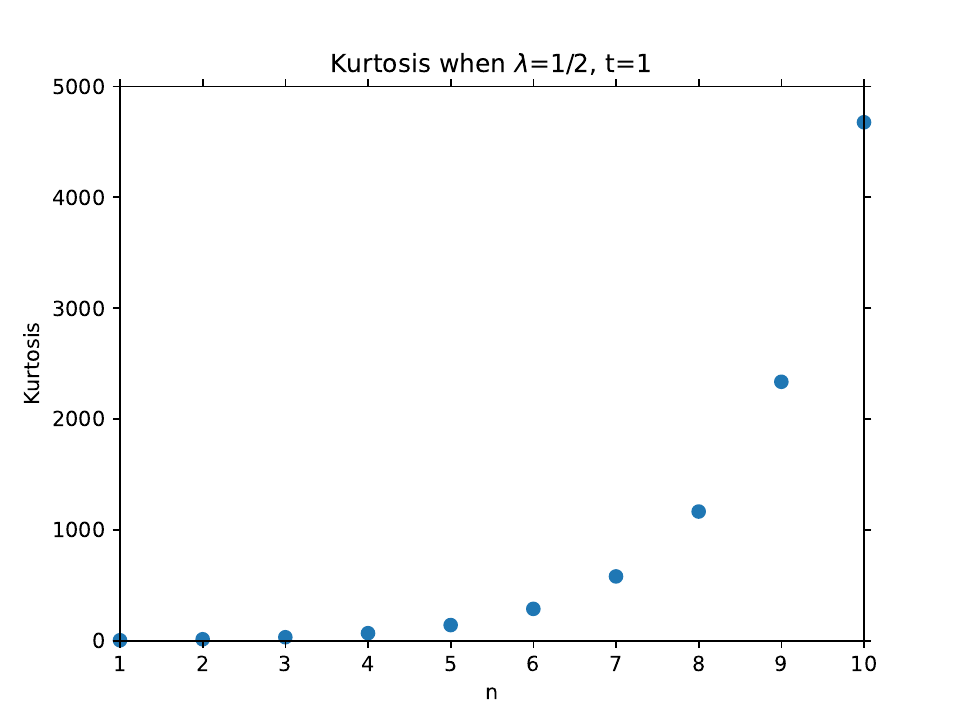}
}
\caption{When $\lambda=\frac{1}{2}$ the moments of MIPP w.r.t. $n$ from 1 to 10}
\end{figure}

\begin{figure}[htb]
\centering
\subfigure[Mean]{
\includegraphics[width=3.8cm]{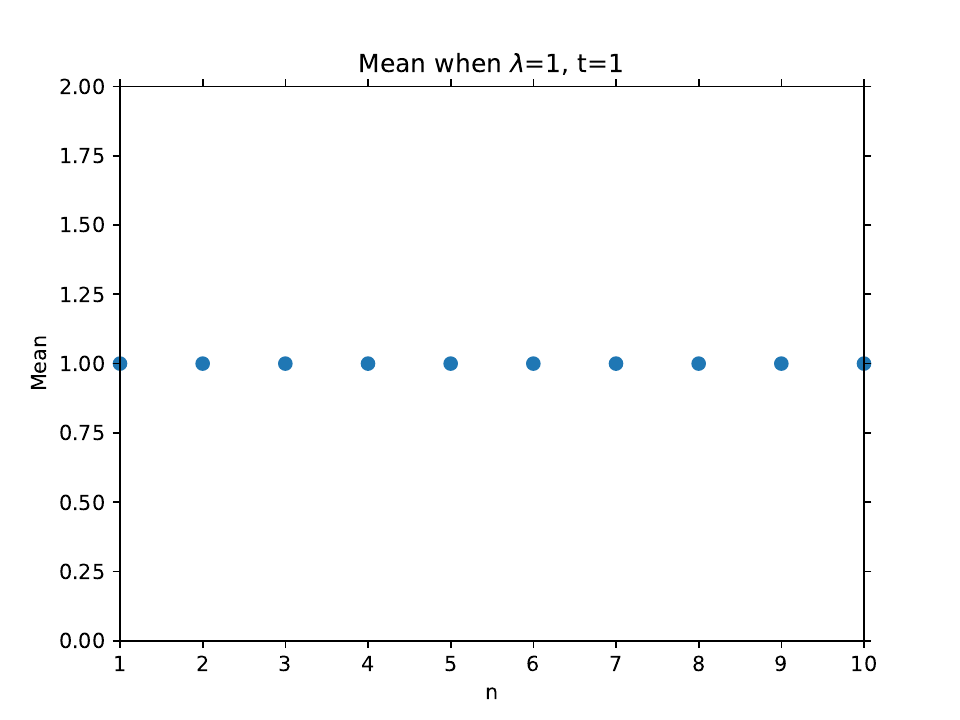}
}
\;
\subfigure[Variance]{
\includegraphics[width=3.8cm]{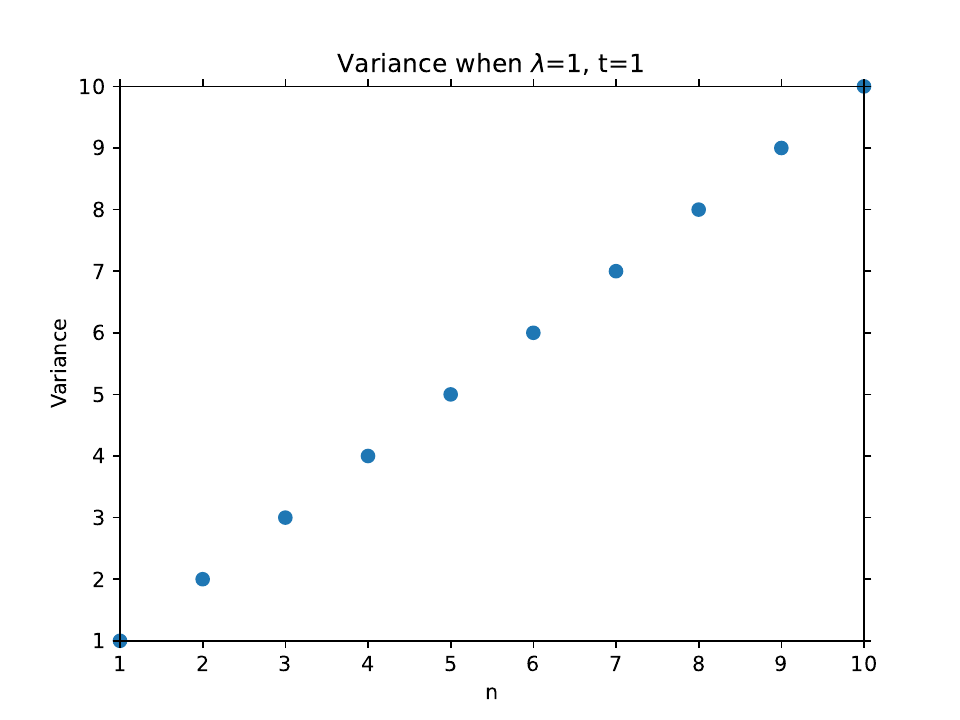}
}
\;
\subfigure[Skewness]{
\includegraphics[width=3.8cm]{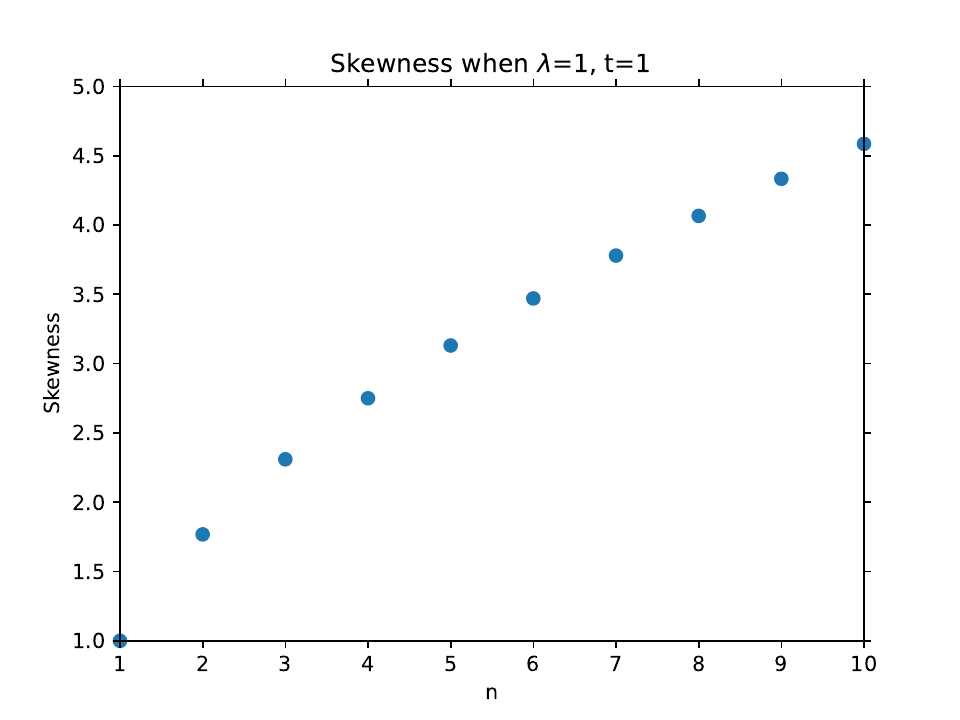}
}
\;
\subfigure[Kurtosis]{
\includegraphics[width=3.8cm]{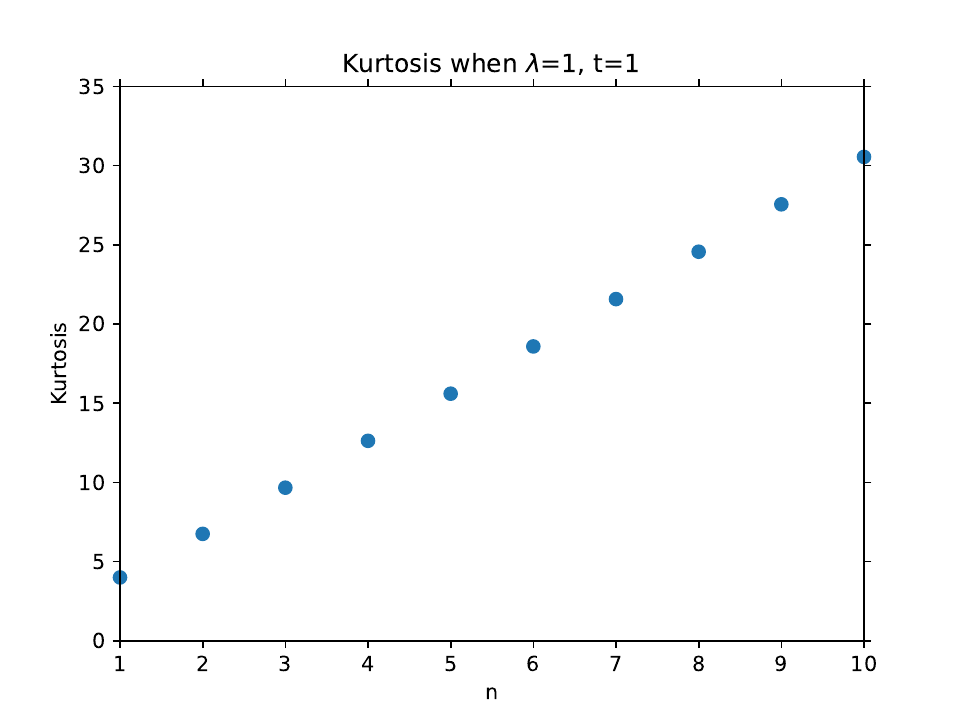}
}
\caption{When $\lambda=1$ the moments of MIPP w.r.t. $n$ from 1 to 10}
\end{figure}

\begin{figure}[htb]
\centering
\subfigure[Mean]{
\includegraphics[width=3.8cm]{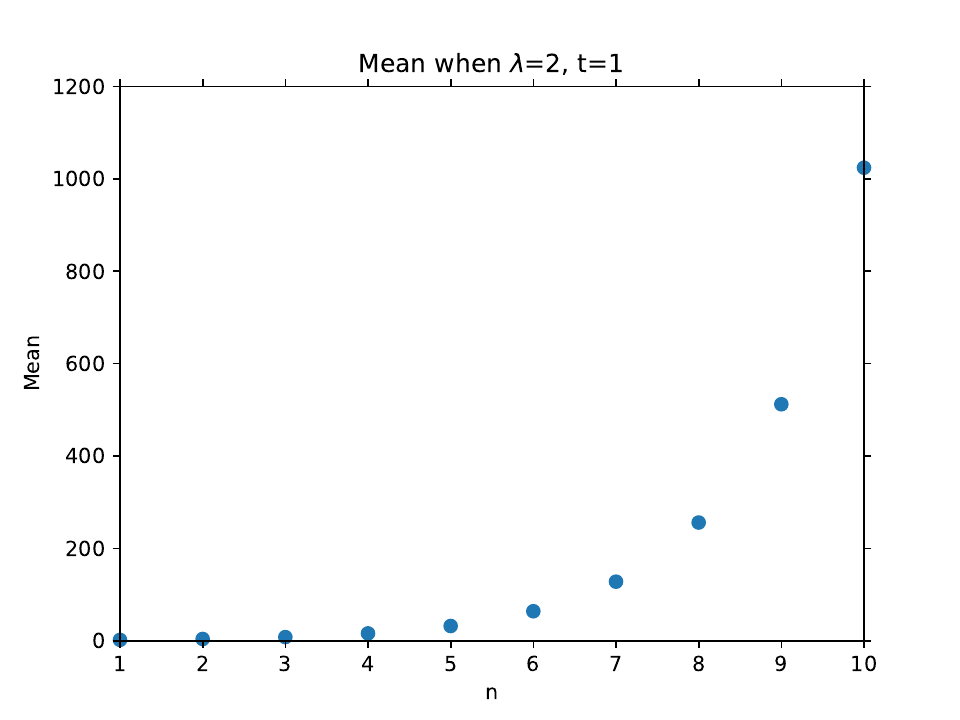}
}
\quad
\subfigure[Variance]{
\includegraphics[width=3.8cm]{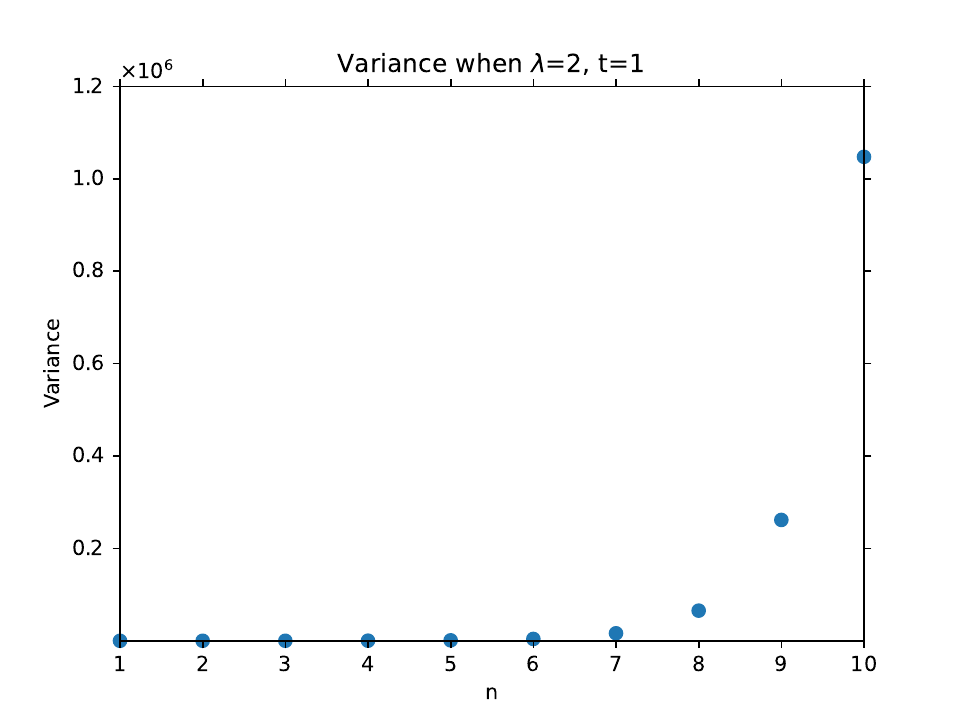}
}
\quad
\subfigure[Skewness]{
\includegraphics[width=3.8cm]{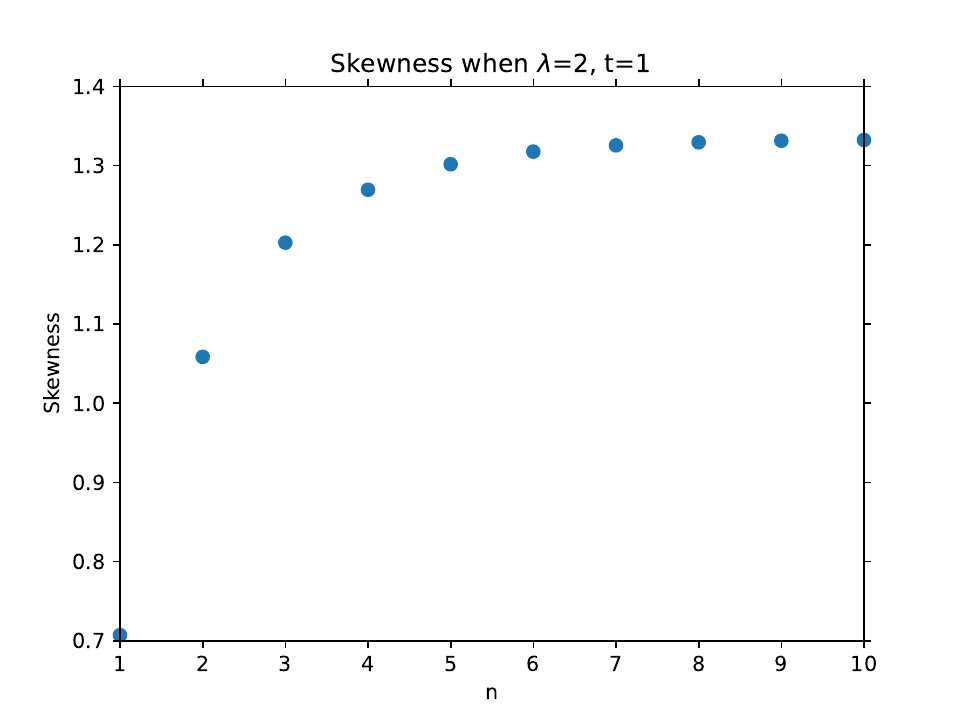}
}
\quad
\subfigure[Kurtosis]{
\includegraphics[width=3.8cm]{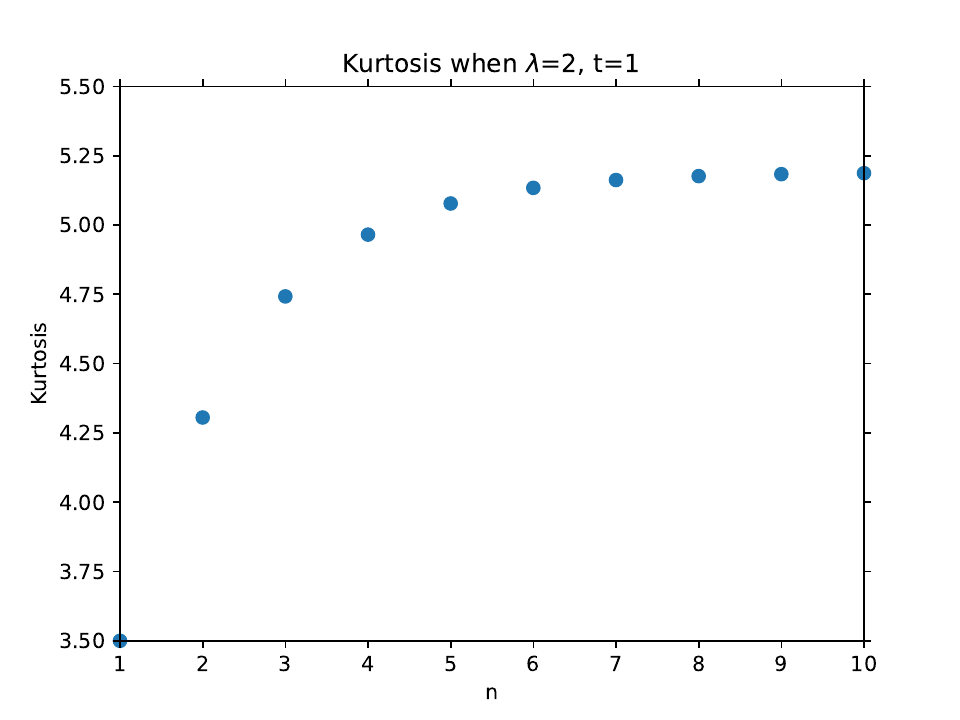}
}
\caption{When $\lambda=2$ the moments of MIPP w.r.t. $n$ from 1 to 10}
\end{figure}

\newpage 

The following two propositions from \cite{Hu-2026} are crucial for developing the results presented in this paper.

\begin{proposition}\label{sojorn} For each fixed $n\geq 2$, the sojourn times of $V_t^{(n)}$ are i.i.d. exponential random variables. More specifically,
\[
S_k^{(n)}\sim \mathrm{EXP}(\lambda q_{n-1}), \forall k\geq 1, 
\]
where $q_{n-1}=P(V_1^{(n-1)}>0)$. The sequence $q_j, j\geq 1,$ satisfy  the  recursive relation
\[
q_j=1-e^{-\lambda q_{j-1}}, \; j\geq 2,
\]
with
\[
q_1=1-e^{-\lambda}.
\]
\end{proposition}

\begin{proposition}\label{jjj} For each integer $k\geq 1$ and for any $n\geq 2$ 
\begin{equation}\label{jtimes}
P(V^{(n)}_{J_1^{(n)}}=k)=\frac{P(V_1^{(n-1)}=k)}{P(V_1^{(n-1)}\geq 1)},
\end{equation}
where $J_1^{(n)}$ is the first jump time of $V_t^{(n)}$.
\end{proposition}

We note that the properties summarized in this section—particularly the exponential distribution of the sojourn times, the recursive form of the probability mass function, and the characteristic exponent—constitute the main analytical toolbox for the ruin-theoretic analysis that follows. Since $V_t^{(n)}$ is a subordinator, the surplus process $X_t^{(n)}(x)$ falls within the extensively studied family of spectrally negative Lévy processes, allowing us to apply strong fluctuation-theoretic tools in the next sections. The recursive representation of the characteristic exponent, $\ell_n(\theta)=\lambda e^{\ell_{n-1}(\theta)}-\lambda$, will be essential for deriving Laplace transforms of the survival probability, whereas the sojourn-time distribution supports the conditioning steps used to obtain the corresponding integro-differential equations. Moreover, the plots of the probability mass function and the moments already hint at a phase transition in the ruin probability, which will be established rigorously in Theorem \ref{th0514}. With these preliminary results in hand, we now proceed to the core objective of the paper: to characterize the ruin probability and to study its dependence on the hierarchical structure of the MIPP.

\section{The  probability of ruin}\label{S:3}
Having established the key probabilistic properties of the MIPP in Section \ref{S:2}, we now study the ultimate ruin probability $\Pi_n(x)$ for the surplus process (\ref{one1}). The aim of this section is twofold: to identify when ruin is certain and to quantify how the hierarchical claim-arrival mechanism alters the classical solvency threshold. We first extend the classical net profit condition \cite{asmussen2010}, \cite{Gerber1979}, \cite{Grandell1991} to the MIPP setting. Because the iteration depth $n$ amplifies the effective claim intensity by $\lambda^n$,  ruin is almost sure whenever $c\le \lambda^nm$, whereas if $c> \lambda^nm$ ruin is no longer certain and the survival probability tends to one as the initial reserve grows; this is stated in Lemma \ref{le1}. We then analyse the limiting behaviour as $n\rightarrow \infty$, which exhibits a phase transition: $\Pi_n(x)\rightarrow 1$ for $\lambda>1$, $\Pi_n(x)\rightarrow m/c$ for $\lambda=1$, and $\Pi_n(x)\rightarrow 0$ for $\lambda<1$ (Theorem \ref{th0514}).  This extends classical ruin theory to clustered and cascading claim arrivals. Related risk models with richer arrival mechanisms include Cox, Hawkes, mixed Poisson, and subordinated or time-changed constructions  \cite{Albrecher2020}, \cite{Klinge2026}, \cite{Palmowski2025}, \cite{selch2016}, \cite{TOMITA-2022}, while the MIPP framework itself is developed in \cite{Hu-2026}. For general ruin-theoretic background, heavy-tailed asymptotics, smoothness and bounds, and fluctuation-theoretic methods, we refer to \cite{Constantinescu2018}, \cite{Kyprianou2014}, \cite{Ramsay2003}, \cite{Usabel2001}. The section closes with numerical illustrations that confirm the phase transition and demonstrate the impact of the iteration depth $n$ and intensity $\lambda$ on the ruin probability.

Recall from the introduction that we have denoted the ruin probability by $\Pi_n$. The associated survival probability is then given by $\Lambda_n(x)=1-\Pi_n(x)$. As shown in \cite{Hu-2026}, the first jump time $J^{(n)}_1$ of the process $V_t^{(n)}$ follows an exponential distribution with mean $1/(\lambda q_{n-1})$, where $q_{n-1}=P(V_1^{(n-1)}>0)$. Using this property, we first establish the following elementary lemma.

\begin{lemma}\label{le1}
i) If $c\le\lambda^n m$, then $\Pi_n(x)=1$ for all $x>0$.

ii) If $c>\lambda^n m$, then $\lim_{x\to+\infty}\Pi_n(x)=0$.
\end{lemma}

\begin{proof}
First observe that the model (\ref{one1}) can be expressed as
\begin{equation*}
\begin{split}
 X^{(n)}_t(x)\overset{d}{=}x+ct-\sum_{i=0}^{N_t}\eta_i,
 \end{split}
\end{equation*}
where $\eta_i\overset{d}{=}\sum_{j=0}^{V_1^{(n-1)}}\xi_j$. Define $\bar{\eta}_i=\eta_i-c(\tau_i-\tau_{i-1})$ and $\tilde{\eta}_n=\sum_{i=1}^n\bar{\eta_i}$, where $\tau_i$ denotes the $i$-th jump time of $N_t$ .
Note that $E[\bar{\eta}_i]=\lambda^{n-1}m-c/\lambda$. It is straightforward to verify that the sequence $\{\bar{\eta}_i\}$ satisfies the hypotheses of Theorem 6.3.1 in \cite{rolski1999}. Let
\[
M_n=\sup_{t\ge0}\Big(\sum_{i=0}^{N_t}\sum_{j=0}^{V_1^{(n-1)}}\xi_j-ct\Big).
\]
and observe that $M_n=\sup_{k\ge1}\tilde{\eta}_k$. Thus
\begin{equation}\label{psi}
\begin{split}
 \Pi_n(x)=P(M_n>x).
\end{split}
\end{equation}

If $c\le\lambda^nm$, then $E[\bar{\eta}_i]\ge0$. By Theorem 6.3.1 in \cite{rolski1999}, this implies $P(\limsup_{k\to\infty}\tilde{\eta}_k=+\infty)=1$, and therefore, using (\ref{psi}), we obtain $\Pi_n(x)=1$.

On the other hand, if $c>\lambda^nm$, then $E[\bar{\eta}_i]<0$, and by Theorem 6.3.1 in \cite{rolski1999} we have $P(\lim_{k\to\infty}\tilde{\eta}_k=-\infty)=1$. Combining this with (\ref{psi}), we conclude that $\lim_{x\to\infty}\Pi_n(x)=0$.
\end{proof}

Next, we investigate the behavior of the ruin probability $\Pi_n(x)$ as $n \rightarrow \infty$ for each fixed $x>0$. Set $f(x)=e^{-1+x}$ and $g(x)=e^{-\lambda+\lambda x}$. The intensity parameter $\lambda$ of the Poisson process $N_t$ is held fixed throughout the paper, so we use the notation $g(x)$ as defined above. It is evident that $f=g$ in the special case $\lambda=1$. The $n-$ folded compositions of $f$ and $g$ are denoted by $f_n^{\circ}$ and $g_n^{\circ}$ respectively.

We can readily verify that the Laplace exponent of $X_t^{(n)}(0)$ takes the form
\begin{equation}\label{scale-9}
\begin{split}
 \phi_n(\theta)=c\theta-\lambda+\lambda f_{n-1}^{\circ}(E[e^{-\xi\theta}]),
\end{split}
\end{equation}
and that the corresponding ruin probability is characterized by
\begin{equation}\label{ruin-10}
\begin{split}
 \Pi_n(x)=1-\phi_n'(0+)W_n(x),
\end{split}
\end{equation}
where $W_n(x)$ denotes the scale function associated with the process $X_t^{(n)}(x)$, see Theorem 3.2 in Chapter XI of \cite{asmussen2010} for this. We first prove the following Lemma.

\begin{lemma}\label{lem0512}
i) For the function $f(x)=e^{-1+x}$, we obtain
\begin{equation*}
\begin{split}
 \lim_{n\to\infty} f_{n}^{\circ}(0)=1,
\end{split}
\end{equation*}
where $f_n^{\circ}(\cdot)$ denotes the $n$-fold self-composition of $f$.

ii) If $\lambda<1$, we have $g_{n}^{\circ}(x)\geq f_n^{\circ}(x)$ when $x\in [0, 1]$ for all $n\geq 1$.
\end{lemma}
\begin{proof} i) 
We first claim that
\begin{equation}\label{05121}
\begin{split}
 f_n^{\circ}(0)=e^{-1}\mathrm{E}_{i=0}^{n-2}e^{e^{-1}},
\end{split}
\end{equation}
where $\mathrm{E}_{i=0}^{n-2}\cdot$ is the continued exponential introduced in \cite{Barrow-1936} and is given by
\begin{equation*}
\begin{split}
 \mathrm{E}_{i=0}^{n-2}a_i=a_0^{a_1^{.^{.^{.^{a_{n-2}}}}}}.
\end{split}
\end{equation*}
We prove this claim by induction. First, when $n=1$, we have
\begin{equation*}
\begin{split}
 f_{1}^{\circ}(0)=f(0)=e^{-1},
\end{split}
\end{equation*}
and therefore (\ref{05121}) is satisfied. When $n=2$, we have
\begin{equation*}
\begin{split}
 f_{2}^{\circ}(0)=e^{-1+e^{-1}}=e^{-1}e^{e^{-1}}=e^{-1}\mathrm{E}_{i=0}^{0}e^{e^{-1}}.
\end{split}
\end{equation*}
Now suppose the relation (\ref{05121})  holds for all $n\le k$ for some positive integer $k$, then 
\begin{equation*}
\begin{split}
 f_{k+1}^{\circ}(0)=e^{-1+f_k^{\circ}(0)}=e^{-1}e^{e^{-1}\mathrm{E}_{i=0}^{k-2}e^{e^{-1}}}=e^{-1}\mathrm{E}_{i=0}^{k-1}e^{e^{-1}},
\end{split}
\end{equation*}
which verifies the claim for $k+1$. We conclude that (\ref{05121}) holds. Next, according to \cite{Barrow-1936}, we have $\lim_{n\to\infty}\mathrm{E}_{i=0}^{n}e^{e^{-1}}=e$. By plugging this into (\ref{05121}), we obtain
\begin{equation*}
\begin{split}
 \lim_{n\to\infty}f_n^{\circ}(0)=e^{-1}\lim_{n\to\infty}\mathrm{E}_{i=0}^{n-2}e^{e^{-1}}=e^{-1}e=1.
\end{split}
\end{equation*}
This completes the proof of i).

ii) First note that $0\le f_n^{\circ}(x)\le 1, 0\le g_n^{\circ}(x)\le 1 $ on $[0, 1]$ for all $n\geq 1$ and $\lambda<1$. We prove the claim in ii) by induction. When $n=1$ we obviously have $e^{-\lambda+\lambda x}\ge e^{-1+x}$. Suppose the claim holds for all $n\le k$. When $n=k+1$ we have 
\begin{equation}\label{compare}
\begin{split}
 g^{\circ}_{k+1}(x)=e^{-\lambda+\lambda g_{k}^{\circ}(x)}\ge e^{-1+g_{k}^{\circ}(x)}\ge e^{-1+f_{k}^{\circ}(x)}=f^{\circ}_{k+1}(x),
\end{split}
\end{equation}
which verifies the claim.
\end{proof}

We use this Lemma in the proof of the following Theorem.
\begin{theorem}\label{th0514}
Assume $c\ge E[\xi]$, then we have
\begin{equation*}
\lim_{n\to\infty}\Pi_n(x)=\left\{
\begin{array}{ll}
0,\quad &\multirow{1}*{$\lambda<1$,}\\
\specialrule{0em}{1ex}{1ex}
\frac{E[\xi]}{c},\quad &\multirow{1}*{$\lambda=1$,}\\
\specialrule{0em}{1ex}{1ex}
1,\quad &\multirow{1}*{$\lambda>1$.}
\end{array}\right.
\end{equation*}
\end{theorem}

\begin{proof}
Case one: When $\lambda>1$, the inequality $\lambda^n m>c$ holds for all sufficiently large $n$. Hence, by part i) of Lemma \ref{le1}, we obtain $\Pi_n(x)=1$ whenever $\lambda^n m>c$. Consequently, $\lim_{n\to\infty}\Pi_n(x)=1$.

Case two: when $\lambda = 1$, we have $\phi'_n(0+)=E X_1^{(n)}(0)=c-\lambda^n m=c-m$, so relation (\ref{ruin-10}) reduces to
\[
 \Pi_n(x)=1-(c-m)W_n(x).
\]
Recall that $\int_0^{\infty} e^{-\theta x} W_n(x)\,dx = \frac{1}{\phi_n(\theta)}$. We first analyze the limit of $\phi_n(\theta)$ as $n\to\infty$. Note that
\begin{equation}\label{psi_n}
\begin{split}
 \phi_n(\theta) = c\theta - 1 + f_{n-1}^{\circ}(E[e^{-\xi\theta}]).
\end{split}
\end{equation}
Thus it is enough to determine $\lim_{n\to\infty} f_{n-1}^{\circ}(E[e^{-\xi\theta}])$. Since $f(x)$ is increasing, each iterate $f_{n-1}^{\circ}(x)$ is also increasing in $x$ for all $n\ge1$. Hence, using $0\le E[e^{-\theta \xi}]\le 1$, we obtain
\begin{equation}\label{smaller}
\begin{split}
f_{n-1}^{\circ}(0) \le f_{n-1}^{\circ}(E[e^{-\xi\theta}]) \le f_{n-1}^{\circ}(1)
= f_{n-2}^{\circ}(1) = \cdots = f(1) = 1.
\end{split}
\end{equation}
Letting $n\to\infty$ in (\ref{smaller}) and applying Lemma \ref{lem0512}, we get
\begin{equation*}
\begin{split}
 \lim_{n\to\infty} f_{n-1}^{\circ}(E[e^{-\xi\theta}]) = 1.
\end{split}
\end{equation*}
Therefore, from (\ref{psi_n}) we deduce
\begin{equation*}
\begin{split}
 \lim_{n\to\infty} \phi_n(\theta) = c\theta.
\end{split}
\end{equation*}
Consequently,
\[
 \lim_{n\to\infty} \int_0^{\infty} e^{-\theta x} W_n(x)\,dx
 = \lim_{n\to\infty} \frac{1}{\phi_n(\theta)}
 = \frac{1}{c\theta},\qquad \theta>0.
\]
The inverse Laplace transform of $1/(c\theta)$ is the constant function $1/c$. Hence $\lim_{n\to\infty} W_n(x) = 1/c$. It follows that
\begin{equation*}
\begin{split}
 \lim_{n\to\infty} \Pi_n(x)
 &= 1 - \lim_{n\to\infty}(\phi_n'(0+) W_n(x)) \\
 &= 1 - (c - E[\xi])\frac{1}{c}
 = \frac{E[\xi]}{c}.
\end{split}
\end{equation*}

Case three: when $\lambda<1$, we obtain
\begin{equation*}
\begin{split}
 \lim_{n\to\infty}\phi'_n(0+)=\lim_{n\rightarrow \infty}EX_1^{(n)}(0)=\lim_{n\to\infty}(c-\lambda^nE[\xi])=c.
\end{split}
\end{equation*}
Moreover, recall that $\phi_n(\theta)=c\theta-\lambda+\lambda g_{n-1}^{\circ}(E[e^{-\xi\theta}])$. Hence
\begin{equation}\label{psi_n_lambda}
\begin{split}
\lim_{n\rightarrow \infty} \phi_n(\theta)=c\theta-\lambda+\lambda \lim_{n\rightarrow \infty}g_{n-1}^{\circ}(E[e^{-\xi\theta}]),
\end{split}
\end{equation}
provided that this limit exists. We now examine $\lim_{n\rightarrow \infty}g_{n-1}^{\circ}(E[e^{-\xi\theta}])$. First note that, for $\lambda<1$, the function $g_n^{\circ}(x)$ is increasing in $x$ for every $n\geq 1$. Consequently,
\begin{equation}\label{smaller_lambda}
\begin{split}
 g_{n-1}^{\circ}(E[e^{-\xi\theta}])\le g_{n-1}^{\circ}(1)=g_{n-2}^{\circ}(1)=\cdots=g(1)=1.
\end{split}
\end{equation}
On the other hand, by part ii) of Lemma \ref{lem0512}, we have
\[
g_n^{\circ}(Ee^{-\xi \theta})\geq f_n^{\circ}(Ee^{-\xi \theta})\geq f_n^{\circ}(0).
\]
The second inequality follows from the monotonicity of $f_n^{\circ}(x)$ together with $0\le Ee^{-\xi \theta}\le 1$. Hence
\[
f_n^{\circ}(0) \le g_n^{\circ}(Ee^{-\xi \theta})\le 1.
\]
Applying part i) of Lemma \ref{lem0512}, we deduce that the limit $\lim_{n\rightarrow \infty}g_{n-1}^{\circ}(E[e^{-\xi\theta}])=1$ exists. Substituting this into (\ref{psi_n_lambda}) yields $\lim_{n\rightarrow \infty}\psi_n(\theta)=c\theta$. Therefore, as in Case 2, we get $\lim_{n\rightarrow \infty}W_n(x)=1/c$. Consequently,
\begin{equation*}
\begin{split}
 \lim_{n\to\infty}\Pi_n(x)=1-\lim_{n\to\infty}(\phi_n'(0+)W_n(x))=1-c\cdot\frac{1}{c}=0.
\end{split}
\end{equation*}
\end{proof}

Next, we examine Theorem \ref{th0514} by plotting the graphs of $\Pi_n(x)$ for different values of $\lambda$ and $n$, while keeping the parameter $c=2$ fixed. The jump size is modeled as an exponential random variable with parameter $\delta=1$. We then employ numerical inversion of the Laplace transform. The resulting plots are consistent with the statement of Theorem \ref{th0514}.

\begin{figure}[htb]
\centering
\subfigure[$\lambda=1/2$]{
\includegraphics[width=8cm]{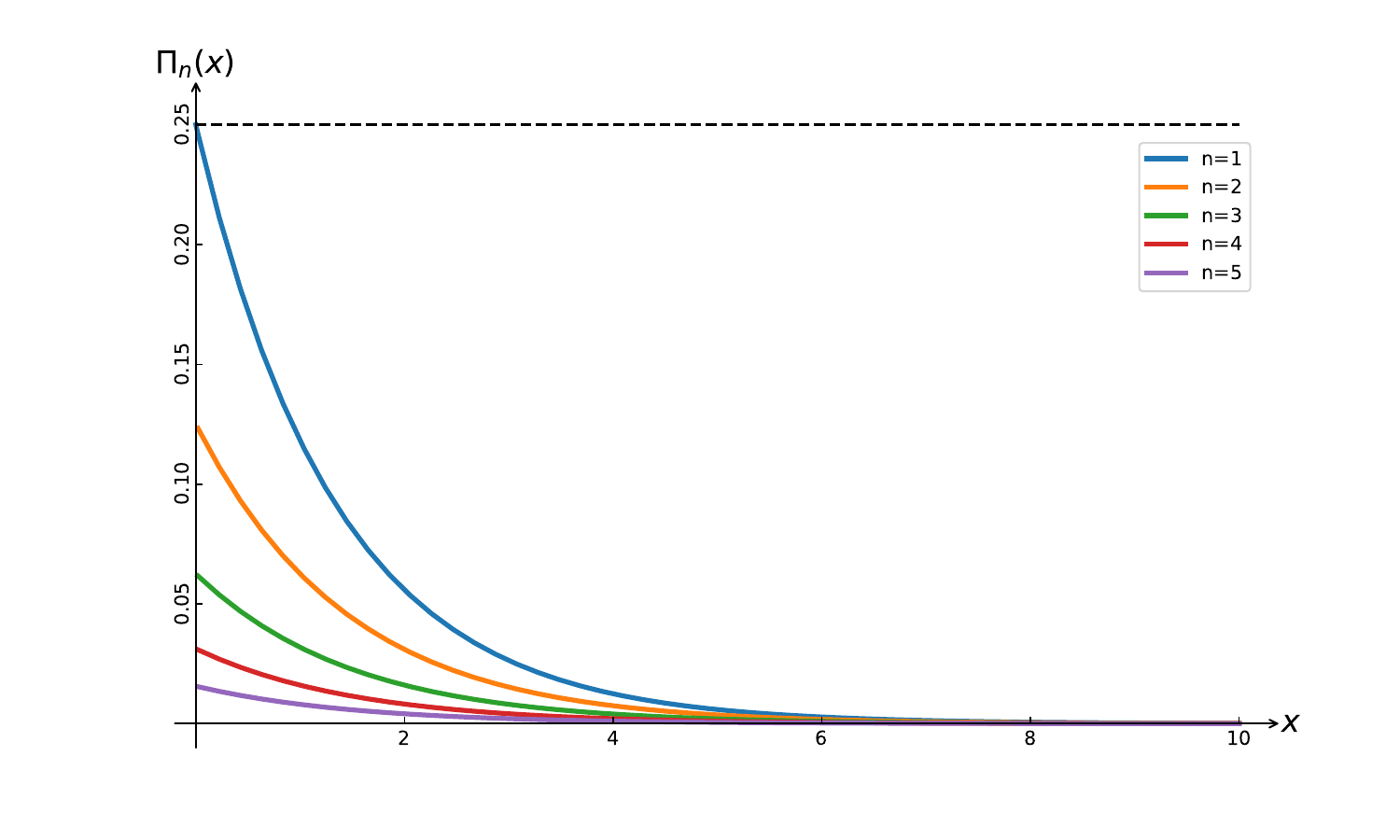}
}
\quad
\subfigure[$\lambda=1$]{
\includegraphics[width=8cm]{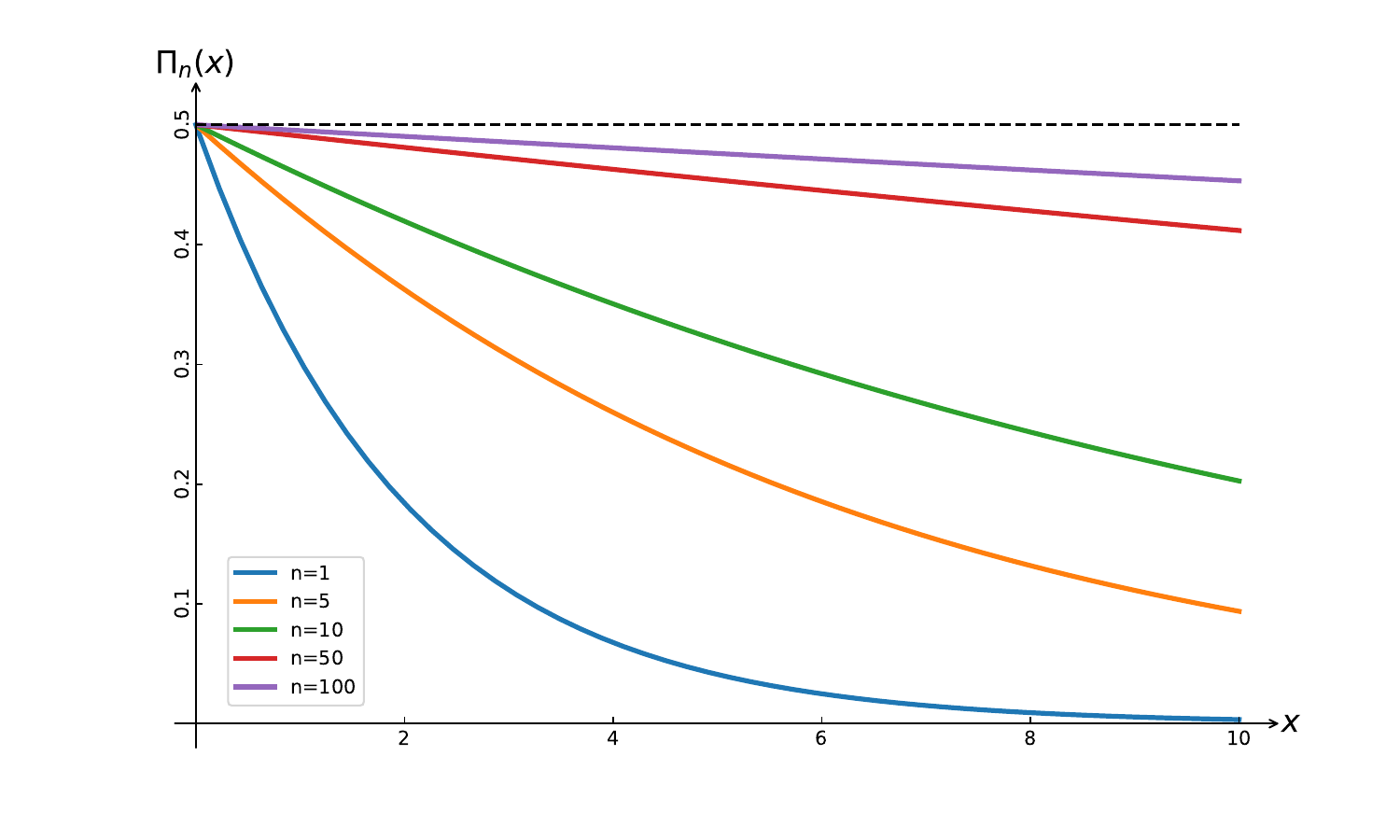}
}
\caption{The ruin probabilities for various values of $\lambda$}
\end{figure}

\newpage

\section{Integro-differential equations associated with ruin probabilities}\label{S:4}

In this section, we focus on the survival probability $\Lambda_n(x)=1-\Pi_n(x)$ for a fixed, finite $n$. Clarifying how it depends on the initial reserve $x$ and on the parameters $(\lambda, n, c, m)$ is crucial for numerical computation and for studying smoothness properties \cite{GERBER1973}, \cite{2016Yuliya}. With this goal, we derive the integro-differential equations that $\Lambda_n(x)$ satisfies. The argument uses the embedding of the MIPP-based surplus process into the classical Crame\'r–Lundberg setting \cite{Albrecher2010},\cite{Dufresne-Gerber1989} and is carried out by conditioning on the first jump time of the outer Poisson process, leading to equations for the right and left derivatives. We then show that these relations can be transformed into a Volterra integral equation of the second kind \cite{Grandell1991},  and, when claim sizes are discrete, into delay differential equations \cite{Polyanin2023}. Comparable results for risk models with more general arrival structures appear in \cite{Palmowski2025}, \cite{TOMITA-2022} . Finally, the section characterizes the jump discontinuity in the derivative of $\Lambda_n(x)$ when the batch-claim distribution has atoms and gives an explicit piecewise solution for the discrete-claim setting.
In the rest of the paper, we use the following notation: for a function $f(x)$, $f^{+}(x)$ denotes the right-hand derivative at $x$, and $f^{-}(x)$ denotes the left-hand derivative at $x$.

Our following Theorem establishes the associated integro-differential equations.
\begin{theorem}\label{th1}
Assume that $c>\lambda^nm$. The survival probability $\Lambda_n(x)$ is continuous on $\mathbb{R}_+$, possessing right and left derivatives denoted by $\Lambda_n^{+}(x)$ and $\Lambda_n^-(x)$, respectively. In addition, we have
\begin{equation}\label{derivatives}
\begin{split}
 c\Lambda_n^{+}(x)&=\lambda q_{n-1}\Lambda_n(x)-\lambda\sum_{k=1}^{\infty}P(V_1^{(n-1)}=k)\int_0^x\Lambda_n(x-y)dF^{*k}(y),\\
 c\Lambda_n^-(x)&=\lambda q_{n-1}\Lambda_n(x)-\lambda\sum_{k=1}^{\infty}P(V_1^{(n-1)}=k)\int_0^{x_-}\Lambda_n(x-y)dF^{*k}(y).
\end{split}
\end{equation}

\end{theorem}

\begin{proof}
First, note that the model can be reduced to the classical Cram\'er–Lundberg framework as it is done in the proof of Lemma \ref{le1}. Therefore continuity, right differentiability, and left differentiability follow directly from Theorem 1.2 in \cite{2016Yuliya} and Theorem 5.3.1 in \cite{rolski1999}.

Next, by invoking Proposition 2 and Proposition 3 in \cite{Hu-2026}, Theorem 1.2 in \cite{2016Yuliya}, and Theorem 5.3.1 in \cite{rolski1999}, we obtain
\begin{equation*}
\begin{split}
 \Lambda_n(x)
 &=e^{-\lambda q_{n-1}t}\Lambda_n(x+ct)
   +\int_0^t \lambda q_{n-1}e^{-\lambda q_{n-1}s}
     \int_0^{x+cs}\Lambda_n(x+cs-y)\,dG_n(y)\,ds,
\end{split}
\end{equation*}
where
\begin{equation*}
\begin{split}
 G_n(y)
 =\sum_{k=1}^{\infty} P\bigl(V_{J^{(n)}_1}^{(n)}=k\bigr)F^{*k}(y)
 =\sum_{k=1}^{\infty}\frac{P\bigl(V_1^{(n-1)}=k\bigr)}{q_{n-1}}\cdot F^{*k}(y),
\end{split}
\end{equation*}
and $J_1^{(n)}$ denotes the first jump time of $V_t^{(n)}$.
This implies
\begin{equation*}
\begin{split}
 c\,\frac{\Lambda_n(x+ct)-\Lambda_n(x)}{ct}
 &=\frac{1-e^{-\lambda q_{n-1}t}}{t}\,\Lambda_n(x+ct)
   -\frac{\displaystyle\int_0^t\lambda q_{n-1}e^{-\lambda q_{n-1}s}
      \int_0^{x+cs}\Lambda_n(x+cs-y)\,dG_n(y)\,ds}{t},
\end{split}
\end{equation*}
and therefore
\begin{equation*}
\begin{split}
 c\Lambda_n^+(x)
 &=\lim_{t\to0}\frac{1-e^{-\lambda q_{n-1}t}}{t}\Lambda_n(x+ct)
   -\lim_{t\to0}\frac{\displaystyle\int_0^t\lambda q_{n-1}e^{-\lambda q_{n-1}s}
      \int_0^{x+cs}\Lambda_n(x+cs-y)\,dG_n(y)\,ds}{t}\\
 &=\lambda q_{n-1}\Lambda_n(x)
   -\lim_{t\to0}\lambda q_{n-1}e^{-\lambda q_{n-1}t}
      \int_0^{x+ct}\Lambda_n(x+ct-y)\,dG_n(y)\\
 &=\lambda q_{n-1}\Lambda_n(x)
   -\lambda q_{n-1}\int_0^x \Lambda_n(x-y)\,dG_n(y)\\
 &=\lambda q_{n-1}\Lambda_n(x)
   -\lambda\sum_{k=1}^{\infty}P\bigl(V_1^{(n-1)}=k\bigr)
      \int_0^x \Lambda_n(x-y)\,dF^{*k}(y).
\end{split}
\end{equation*}

For the left derivative, we proceed analogously:
\begin{equation*}
\begin{split}
 \Lambda_n(x-ct)
 &=e^{-\lambda q_{n-1}t}\Lambda_n(x)
   +\int_0^t\lambda q_{n-1}e^{-\lambda q_{n-1}s}
     \int_0^{x-c(t-s)}\Lambda_n(x-c(t-s)-y)\,dG_n(y)\,ds,
\end{split}
\end{equation*}
which yields
\begin{equation*}
\begin{split}
 c\,\frac{\Lambda_n(x-ct)-\Lambda_n(x)}{ct}
 &=\frac{e^{-\lambda q_{n-1}t}-1}{t}\,\Lambda_n(x)
   +\frac{\displaystyle\int_0^t\lambda q_{n-1}e^{-\lambda q_{n-1}s}
      \int_0^{x-c(t-s)}\Lambda_n(x-c(t-s)-y)\,dG_n(y)\,ds}{t},
\end{split}
\end{equation*}
and hence
\begin{equation*}
\begin{split}
 c\Lambda_n^-(x)
 &=\lim_{t\to0}\frac{1-e^{-\lambda q_{n-1}t}}{t}\Lambda_n(x)
   -\lim_{t\to0}\frac{\displaystyle\int_0^t\lambda q_{n-1}e^{-\lambda q_{n-1}s}
      \int_0^{x-c(t-s)}\Lambda_n(x-c(t-s)-y)\,dG_n(y)\,ds}{t}\\
 &=\lambda q_{n-1}\Lambda_n(x)
   -\lambda q_{n-1}\int_0^{x_-}\Lambda_n(x-y)\,dG_n(y)\\
 &=\lambda q_{n-1}\Lambda_n(x)
   -\lambda\sum_{k=1}^{\infty}P\bigl(V_1^{(n-1)}=k\bigr)
      \int_0^{x_-}\Lambda_n(x-y)\,dF^{*k}(y).
\end{split}
\end{equation*}
\end{proof}

\begin{remark} (An alternative derivation)
As an alternative approach to proving Theorem \ref{th1}, we can express the process in the form
\begin{equation*}
\begin{split}
 X^{(n)}_t(x)=x+ct-\sum_{i=0}^{N_t}\Big(\sum_{j=0}^{V_1^{(n-1)}}\xi_j\Big),
\end{split}
\end{equation*}
following Lemma \ref{le1}. For a short time interval $t$, we consider the first jump time $s$ of $N_t$ and split the dynamics into three disjoint scenarios: i) $s>t$, so $N_t$ has no jumps on $[0,t]$;
ii) $s\le t$, but $V_1^{(n-1)}$ does not jump, i.e., $N_t$ jumps while $V_1^{(n-1)}$ remains unchanged;
iii) $s\le t$ and at the same time $V_1^{(n-1)}$ jumps, i.e., both $N_t$ and $V_1^{(n-1)}$ jump.
We treat these cases one by one.

i) On the event that $N_t$ has no jump, the survival probability is
\begin{equation*}
\begin{split}
 P(N_t=0)\Lambda_n(x+ct)=e^{-\lambda t}\Lambda_n(x+ct).
\end{split}
\end{equation*}

ii) On the event that $N_t$ jumps but $V_1^{(n-1)}$ does not, and conditioning on the first jump time $s$, the survival probability is given by
\begin{equation*}
\begin{split}
 \int_0^t\lambda e^{-\lambda s}P(V_1^{(n-1)}=0)\Lambda_n(x+cs)\,ds
  =(1-q_{n-1})\int_0^t\lambda e^{-\lambda s}\Lambda_n(x+cs)\,ds.
\end{split}
\end{equation*}

iii) If both $N_t$ and $V_1^{(n-1)}$ jump, then, conditioning again on the first jump time $s$, the survival probability becomes
\begin{equation*}
\begin{split}
 &\int_0^t\lambda e^{-\lambda s}\sum_{k=1}^{\infty}P(V_1^{(n-1)}=k)\int_0^{x+cs}\Lambda_n(x+cs-y)\,dF^{*k}(y)\,ds\\
 =&\sum_{k=1}^{\infty}P(V_1^{(n-1)}=k)\int_0^t\lambda e^{-\lambda s}\int_0^{x+cs}\Lambda_n(x+cs-y)\,dF^{*k}(y)\,ds.
\end{split}
\end{equation*}
Putting the three contributions together, we obtain
\begin{equation*}
\begin{split}
 \Lambda_n(x)&=e^{-\lambda t}\Lambda_n(x+ct)
 +(1-q_{n-1})\int_0^t\lambda e^{-\lambda s}\Lambda_n(x+cs)\,ds\\
 &\quad+\sum_{k=1}^{\infty}P(V_1^{(n-1)}=k)\int_0^t\lambda e^{-\lambda s}\int_0^{x+cs}\Lambda_n(x+cs-y)\,dF^{*k}(y)\,ds.
\end{split}
\end{equation*}
Next, we evaluate the right derivative:
\begin{equation*}
\begin{split}
 c\Lambda_n^+(x)&=\lim_{t\to0}\frac{(1-e^{-\lambda t})\Lambda_n(x+ct)}{t}
 -(1-q_{n-1})\lim_{t\to0}\frac{\int_0^t\lambda e^{-\lambda s}\Lambda_n(x+cs)\,ds}{t}\\
 &\quad-\sum_{k=1}^{\infty}P(V_1^{(n-1)}=k)\lim_{t\to0}\frac{\int_0^t\lambda e^{-\lambda s}\int_0^{x+cs}\Lambda_n(x+cs-y)\,dF^{*k}(y)\,ds}{t}\\
 &=\lambda q_{n-1}\Lambda_n(x)-\lambda\sum_{k=1}^{\infty}P(V_1^{(n-1)}=k)\int_0^x\Lambda_n(x-y)\,dF^{*k}(y),
\end{split}
\end{equation*}
which matches (\ref{derivatives}). Similarly, starting from
\begin{equation*}
\begin{split}
 \Lambda_n(x-ct)&=e^{-\lambda t}\Lambda_n(x)
 +(1-q_{n-1})\int_0^t\lambda e^{-\lambda s}\Lambda_n(x-c(t-s))\,ds\\
 &\quad+\sum_{k=1}^{\infty}P(V_1^{(n-1)}=k)\int_0^t\lambda e^{-\lambda s}\int_0^{x-c(t-s)}\Lambda_n(x-c(t-s)-y)\,dF^{*k}(y)\,ds,
\end{split}
\end{equation*}
we derive the left derivative
\begin{equation*}
\begin{split}
 c\Lambda_n^-(x)&=\lim_{t\to0}\frac{(1-e^{-\lambda t})\Lambda_n(x)}{t}
 -(1-q_{n-1})\lim_{t\to0}\frac{\int_0^t\lambda e^{-\lambda s}\Lambda_n(x-c(t-s))\,ds}{t}\\
 &\quad-\sum_{k=1}^{\infty}P(V_1^{(n-1)}=k)\lim_{t\to0}\frac{\int_0^t\lambda e^{-\lambda s}\int_0^{x-c(t-s)}\Lambda_n(x-c(t-s)-y)\,dF^{*k}(y)\,ds}{t}\\
 &=\lambda q_{n-1}\Lambda_n(x)-\lambda\sum_{k=1}^{\infty}P(V_1^{(n-1)}=k)\int_0^{x_-}\Lambda_n(x-y)\,dF^{*k}(y),
\end{split}
\end{equation*}
which is, once more, the same as (\ref{derivatives}).
\end{remark}

While Theorem \ref{th1} provides the integro-differential equation for $\Lambda_n(x)$ and shows that its one-sided derivatives always exist, the smoothness of the survival probability depends on the nature of the batch-claim distribution $G_n$. When $G_n$ has an atom at point $x_k,$ the derivative of $\Lambda_n$  generally experiences a jump there. The following result quantifies this discontinuity exactly, expressing the size of the jump in terms of the atom $p_k$, the initial survival probability $\Lambda_n(0)$, and the model parameters; see \cite{GERBER1973}, \cite{2016Yuliya} for related smoothness results in risk theory.

\begin{theorem}
Suppose $x_k$ denotes the $k$th atom (discrete point) of $G_n(y)$; that is, for some $k\in\mathbb{N}_+$ we have $G_n(x_k)-G_n(x_{k-})=p_k>0$. Then
\begin{equation}\label{deritive_difference}
\begin{split}
 \Lambda_n^-(x_k)-\Lambda_n^+(x_k)=\frac{\lambda q_{n-1} p_k\Lambda_n(0)}{c}=\frac{\lambda q_{n-1}p_k}{c}\left(1-\frac{\lambda^nm}{c}\right).
\end{split}
\end{equation}
\end{theorem}

\begin{proof}
Because $\Lambda_n(x)$ is continuous on $\mathbb{R}_+$, and
\begin{equation*}
\begin{split}
 c\Lambda_n^+(x)&=\lambda q_{n-1}\Lambda_n(x)-\lambda q_{n-1}\int_0^x\Lambda_n(x-y)\,dG_n(y),\quad x\ge0,\\
 c\Lambda_n^-(x)&=\lambda q_{n-1}\Lambda_n(x)-\lambda q_{n-1}\int_0^{x_-}\Lambda_n(x-y)\,dG_n(y),\quad x>0,
\end{split}
\end{equation*}
we proceed by decomposing $G_n$ into its continuous and discrete components. Write $G_n=G_n^{\mathrm{c}}+G_n^{\mathrm{d}}$, where $G_n^{\mathrm{c}}$ is the continuous part and $G_n^{\mathrm{d}}$ is the discrete part. Correspondingly, $\Lambda_n(x)$ can be expressed as $\Lambda_n(x)=\Lambda_n^{\mathrm{c}}(x)+\Lambda_n^{\mathrm{d}}(x)$, where
\begin{equation*}
\begin{split}
 c\Lambda_n^{\mathrm{c}+}(x)&=\lambda q_{n-1}\Lambda_n^{\mathrm{c}}(x)-\lambda q_{n-1}\int_0^x\Lambda_n(x-y)\,dG_n^{\mathrm{c}}(y),\quad x\ge0,\\
 c\Lambda_n^{\mathrm{c}-}(x)&=\lambda q_{n-1}\Lambda_n^{\mathrm{c}}(x)-\lambda q_{n-1}\int_0^{x_-}\Lambda_n(x-y)\,dG_n^{\mathrm{c}}(y),\quad x>0,\\
 c\Lambda_n^{\mathrm{d}+}(x)&=\lambda q_{n-1}\Lambda_n^{\mathrm{d}}(x)-\lambda q_{n-1}\int_0^x\Lambda_n(x-y)\,dG_n^{\mathrm{d}}(y)\\
 &=\lambda q_{n-1}\Lambda_n^{\mathrm{d}}(x)-\lambda q_{n-1}\sum_{x_k\le x}\Lambda_n(x-x_k)p_k,\quad x\ge0,\\
 c\Lambda_n^{\mathrm{d}-}(x)&=\lambda q_{n-1}\Lambda_n^{\mathrm{d}}(x)-\lambda q_{n-1}\int_0^{x_-}\Lambda_n(x-y)\,dG_n^{\mathrm{d}}(y)\\
 &=\lambda q_{n-1}\Lambda_n^{\mathrm{d}}(x)-\lambda q_{n-1}\sum_{x_k< x}\Lambda_n(x-x_k)p_k,\quad x>0.
\end{split}
\end{equation*}

Note that
\begin{equation*}
\begin{split}
 \Lambda_n^{\mathrm{c}-}(x)-\Lambda_n^{\mathrm{c}+}(x)=\frac{\lambda q_{n-1}}{c}\int_{x_-}^x\Lambda_n(x-y)\,dG_n^{\mathrm{c}}(y).
\end{split}
\end{equation*}
Moreover,
\begin{equation*}
\begin{split}
 0\le\int_{x-}^x\Lambda_n(x-y)\,dG_n^{\mathrm{c}}(y)\le\int_{x-}^x dG_n^{\mathrm{c}}(y)
 =G_n^{\mathrm{c}}(x)-\lim_{z\uparrow x}G_n^{\mathrm{c}}(z)=0,
\end{split}
\end{equation*}
which implies $\Lambda_n^{\mathrm{c}-}(x)=\Lambda_n^{\mathrm{c}+}(x)$ for every $x\in\mathbb{R}_+$.

For the discrete part, we have
\begin{equation*}
\begin{split}
 \Lambda_n^{\mathrm{d}-}(x_k)-\Lambda_n^{\mathrm{d}+}(x_k)=\frac{\lambda q_{n-1}}{c}\Lambda_n(0)p_k.
\end{split}
\end{equation*}
Consequently,
\begin{equation*}
\begin{split}
 \Lambda_n^{-}(x_k)-\Lambda_n^{+}(x_k)=\frac{\lambda q_{n-1}}{c}\Lambda_n(0)p_k
 =\frac{\lambda q_{n-1}p_k}{c}\left(1-\frac{\lambda^nm}{c}\right).
\end{split}
\end{equation*}
\end{proof}

The ruin probability $\Pi_n(x)$ similarly satisfies the following integro-differential equation.

\begin{theorem}\label{4.33}
Assume $c>\lambda^nm$. Then the ruin probability $\Pi_n(x)$ satisfies the integral equation
\begin{equation*}
\begin{split}
 c\Pi_n(x)=\lambda\sum_{k=1}^{\infty}P(V_1^{(n-1)}=k)\left(\int_x^{\infty}(1-F^{*k}(y))\,dy+\int_0^x\Pi_n(x-y)(1-F^{*k}(y))\,dy\right).
\end{split}
\end{equation*}
\end{theorem}

\begin{proof}
Integrating the first identity in (\ref{derivatives}) yields
\begin{equation}\label{2}
\begin{split}
 c(\Lambda_n(x)-\Lambda_n(0))&=\lambda q_{n-1}\int_0^x\Lambda_n(z)\,dz-\lambda\sum_{k=1}^{\infty}P(V_1^{(n-1)}=k)\int_0^x\int_0^z\Lambda_n(z-y)\,dF^{*k}(y)\,dz\\
 &=\lambda q_{n-1}\int_0^x\Lambda_n(z)\,dz-\lambda\sum_{k=1}^{\infty}P(V_1^{(n-1)}=k)\int_0^x\Lambda_n(z)F^{*k}(x-z)\,dz\\
 &=\lambda\sum_{k=1}^{\infty}P(V_1^{(n-1)}=k)\int_0^x\Lambda_n(z)\bigl(1-F^{*k}(x-z)\bigr)\,dz\\
 &=\lambda\sum_{k=1}^{\infty}P(V_1^{(n-1)}=k)\int_0^x\Lambda_n(x-y)\bigl(1-F^{*k}(y)\bigr)\,dy.
\end{split}
\end{equation}
Applying the monotone convergence theorem and letting $x\to+\infty$, we obtain
\begin{equation*}
\begin{split}
 c\bigl(\lim_{x\to+\infty}\Lambda_n(x)-\Lambda_n(0)\bigr)&=\lambda\sum_{k=1}^{\infty}P(V_1^{(n-1)}=k)\int_0^{+\infty}\bigl(1-F^{*k}(y)\bigr)\,dy\\
 &=\lambda\sum_{k=1}^{\infty}P(V_1^{(n-1)}=k)\,km=\lambda m E[V_1^{(n-1)}]\\
 &=\lambda^nm,
\end{split}
\end{equation*}
which gives
\begin{equation*}
\begin{split}
 \Lambda_n(0)=1-\frac{\lambda^nm}{c},
\end{split}
\end{equation*}
as $\lim_{x\rightarrow +\infty}\Lambda_n(x)=1-\lim_{x\rightarrow +\infty}\Pi_n(x)=1$ due to part ii) of Lemma \ref{le1}. Substituting $\Lambda_n(x)=1-\Pi_n(x)$ into (\ref{2}), we arrive at
\begin{equation*}
\begin{split}
 c\Pi_n(x)&=\lambda^nm-\lambda\sum_{k=1}^{\infty}P(V_1^{(n-1)}=k)\left(\int_0^x\bigl(1-F^{*k}(y)\bigr)\,dy-\int_0^x\Pi_n(x-y)\bigl(1-F^{*k}(y)\bigr)\,dy\right)\\
 &=\lambda\sum_{k=1}^{\infty}P(V_1^{(n-1)}=k)\left(\int_x^{\infty}\bigl(1-F^{*k}(y)\bigr)\,dy+\int_0^x\Pi_n(x-y)\bigl(1-F^{*k}(y)\bigr)\,dy\right).
\end{split}
\end{equation*}
\end{proof}

\begin{remark} (Volterra Integral Equation) Recall that in the proof of Theorem \ref{4.33} we established that $\Lambda_n(0)=1-\frac{\lambda^nm}{c}$. From the first equation in (\ref{derivatives}), we obtain 
\begin{equation*}
\begin{split}
 \Lambda_n^+(x)&=\frac{\lambda q_{n-1}}{c}\Lambda_n(x)-\frac{\lambda}{c}\sum_{k=1}^{\infty}P(V_1^{(n-1)}=k)\int_0^x\Lambda_n(x-y)dF^{*k}(y)\\
 &=\frac{\lambda q_{n-1}}{c}\Lambda_n(x)-\frac{\lambda q_{n-1}}{c}\int_0^x\Lambda_{n}(x-y)dG_n(y),
\end{split}
\end{equation*}
and hence
\begin{equation*}
\begin{split}
 \left(e^{-\frac{\lambda q_{n-1}}{c}x}\Lambda_n(x)\right)_+'=-\frac{\lambda q_{n-1}}{c}e^{-\frac{\lambda q_{n-1}}{c}x}\int_0^x\Lambda_{n}(x-y)dG_n(y).
\end{split}
\end{equation*}
Integrating both sides from $0$ to $x$ yields
\begin{equation*}
\begin{split}
 e^{-\frac{\lambda q_{n-1}}{c}x}\Lambda_n(x)-\Lambda_n(0)&=-\frac{\lambda q_{n-1}}{c}\int_0^xe^{-\frac{\lambda q_{n-1}}{c}z}\int_0^z\Lambda_{n}(z-y)dG_n(y)dz\\
 &=-\frac{\lambda q_{n-1}}{c}\int_0^x\int_y^xe^{-\frac{\lambda q_{n-1}}{c}z}\Lambda_{n}(z-y)dzdG_n(y)\\
 &\overset{t=z-y}{=}-\frac{\lambda q_{n-1}}{c}\int_0^x\int_0^{x-y}e^{-\frac{\lambda q_{n-1}}{c}(y+t)}\Lambda_{n}(t)dtdG_n(y)\\
 &=-\frac{\lambda q_{n-1}}{c}\int_0^x\Lambda_n(t)\int_0^{x-t}e^{-\frac{\lambda q_{n-1}}{c}(y+t)}dG_n(y)dt,
\end{split}
\end{equation*}
which, together with $\Lambda_n(0)=1-\frac{\lambda^nm}{c}$, leads to
\begin{equation}\label{volterra}
\begin{split}
 \Lambda_n(x)=(1-\frac{\lambda^nm}{c})e^{\frac{\lambda q_{n-1}}{c}x}-\frac{\lambda q_{n-1}}{c}\int_0^x\Lambda_n(t)\int_0^{x-t}e^{-\frac{\lambda q_{n-1}}{c}(y+t-x)}dG_n(y)dt.
\end{split}
\end{equation}
By \cite{Polyanin}, equation (\ref{volterra}) is a classical Volterra integral equation of the second kind with kernel
\begin{equation*}
\begin{split}
 K(x,t)=-\frac{\lambda q_{n-1}}{c}\int_0^{x-t}e^{-\frac{\lambda q_{n-1}}{c}(y+t-x)}dG_n(y).
\end{split}
\end{equation*}
Observe that $K(x,t)$ depends only on the difference $x-t$, i.e., $K(x,t)=K(x-t)$, so it is a difference kernel. For any Volterra equation of the second kind with a locally integrable kernel $K$, there exists a unique solution $\Lambda(x)$ that is locally bounded and continuous on $[0,\infty)$. This is consistent with corresponding Volterra integral equation representations of the ruin probability in Poisson arrival models; see \cite{Gerber1979} and \cite{Seal1969}.
\end{remark}

\subsection{The case where the jump size is held constant: $\xi=m$}

As an illustration, consider the most straightforward situation: $\xi=m$. Assume that all claim sizes are identical and equal to some constant $m>0$. In this case, the $k$-fold convolution of $F$ is given by $F^{*k}(y)=1_{\{y\ge km\}}$, which is a step function. The integral in (\ref{derivatives}) therefore reduces to a sum over discrete points:
\begin{equation*}
\begin{split}
 \int_0^x\Lambda_n(x-y)\,dF^{*k}(y)= \int_0^{x-}\Lambda_n(x-y)\,dF^{*k}(y)=\Lambda_n(x-km)\cdot1_{\{x\ge km\}},\quad x\neq km.
\end{split}
\end{equation*}

As a result, the original integro-differential equation simplifies to a delay differential equation:
\begin{equation*}
\begin{split}
 c\Lambda_n'(x)=\lambda q_{n-1}\Lambda_n(x)-\lambda\sum_{k=1}^{\lfloor x/m\rfloor}P(V_1^{(n-1)}=k)\Lambda_n(x-km),\quad x\neq km.
\end{split}
\end{equation*}
This delay equation can then be solved iteratively on each interval $[km,(k+1)m)$ by applying the method of steps in \cite{Polyanin2023}, which yields an explicit piecewise solution consisting of exponential and polynomial terms as follows.

\begin{theorem}\label{th12_discrete}
The survival probability can be written as
\begin{equation}\label{survival_discrete}
\begin{split}
 \Lambda_n(x)=\sum_{k=0}^{\infty}\Lambda_{n,k}(x),
\end{split}
\end{equation}
where, for $k\ge0$,
\begin{equation*}
\begin{split}
 \Lambda_{n,k}(x)&=\left(\Lambda_n(km)e^{-\frac{\lambda q_{n-1}}{c}km}-\frac{\lambda}{c}\sum_{i=1}^kP(V_1^{(n-1)}=i)\int_{km}^x\Lambda_n(y-im)e^{-\frac{\lambda q_{n-1}}{c}y}\,dy\right)\\
 &\times e^{\frac{\lambda q_{n-1}}{c}x}\cdot1_{\{x\in[km,(k+1)m)\}} .
\end{split}
\end{equation*}
Moreover,
\begin{equation*}
\begin{split}
 \Lambda_n(0)=1-\frac{\lambda^nm}{c}.
\end{split}
\end{equation*}
\end{theorem}

\begin{proof}
We establish (\ref{survival_discrete}) using induction on the interval index.

i) For $x\in[0,m)$, $\Lambda_n$ satisfies
\begin{equation*}
\begin{split}
 \Lambda_n'(x)-\frac{\lambda q_{n-1}}{c}\Lambda_n(x)=0,
\end{split}
\end{equation*}
together with the initial value
\begin{equation*}
\begin{split}
 \Lambda_n(0)=1-\frac{\lambda^nm}{c}.
\end{split}
\end{equation*}
Solving this first-order ODE yields
\begin{equation*}
\begin{split}
 \Lambda_n(x)=(1-\frac{\lambda^nm}{c})e^{\frac{\lambda q_{n-1}}{c}x}
 =\Lambda_n(0)e^{\frac{\lambda q_{n-1}}{c}x}
 =\Lambda_{n,0}(x),\quad x\in[0,m).
\end{split}
\end{equation*}

Next, when $x\in[m,2m)$, we have
\begin{equation*}
\begin{split}
 \Lambda_n'(x)-\frac{\lambda q_{n-1}}{c}\Lambda_n(x)=-\frac{\lambda}{c}P(V_1^{(n-1)}=1)\Lambda_n(x-m),
\end{split}
\end{equation*}
and the boundary condition at $x=m$ is
\begin{equation*}
\begin{split}
 \Lambda_n(m)=\Lambda_n(0)e^{\frac{\lambda  q_{n-1}}{c}m}.
\end{split}
\end{equation*}
By the standard method for linear ODEs,
\begin{equation*}
\begin{split}
 \Lambda_n(x)=\left(-\frac{\lambda}{c}P(V_1^{(n-1)}=1)\int_{m}^x\Lambda_n(y-m)e^{-\frac{\lambda q_{n-1}}{c}y}\,dy+C_1\right)e^{\frac{\lambda q_{n-1}}{c}x},
\end{split}
\end{equation*}
with constant $C_1$. Imposing the initial condition gives
\begin{equation*}
\begin{split}
 C_1=\Lambda_n(0)=\Lambda_n(m)e^{-\frac{\lambda q_{n-1}}{c}m}.
\end{split}
\end{equation*}
Therefore,
\begin{equation*}
\begin{split}
 \Lambda_n(x)&=\left(\Lambda_n(m)e^{-\frac{\lambda q_{n-1}}{c}m}-\frac{\lambda}{c}P(V_1^{(n-1)}=1)\int_{m}^x\Lambda_n(y-m)e^{-\frac{\lambda q_{n-1}}{c}y}\,dy\right)e^{\frac{\lambda q_{n-1}}{c}x}\\
 &=\Lambda_{n,1}(x),\qquad x\in[m,2m).
\end{split}
\end{equation*}

ii) Assume now that (\ref{survival_discrete}) is valid for $k-1$, i.e., for every $x\in[(k-1)m,km)$,
\begin{equation*}
\begin{split}
 \Lambda_n(x)&=\Lambda_{n,k-1}(x)\\
 &=\left(\Lambda_n((k-1)m)e^{-\frac{\lambda q_{n-1}}{c}(k-1)m}-\frac{\lambda}{c}\sum_{i=1}^{k-1}P(V_1^{(n-1)}=i)\int_{(k-1)m}^x\Lambda_n(y-im)e^{-\frac{\lambda q_{n-1}}{c}y}\,dy\right)e^{\frac{\lambda q_{n-1}}{c}x}.
\end{split}
\end{equation*}
For $x\in[km,(k+1)m)$, $\Lambda_n$ then satisfies
\begin{equation*}
\begin{split}
 \Lambda_n'(x)-\frac{\lambda q_{n-1}}{c}\Lambda_n(x)=-\frac{\lambda}{c}\sum_{i=1}^kP(V_1^{(n-1)}=i)\Lambda_n(x-im),
\end{split}
\end{equation*}
with initial value at $x=km$ given by
\begin{equation*}
\begin{split}
 \Lambda_n(km)&=\left(\Lambda_n((k-1)m)e^{-\frac{\lambda q_{n-1}}{c}(k-1)m}
 -\frac{\lambda}{c}\sum_{i=1}^{k-1}P(V_1^{(n-1)}=i)\int_{(k-1)m}^{km}\Lambda_n(y-im)e^{-\frac{\lambda q_{n-1}}{c}y}\,dy\right)e^{\frac{\lambda q_{n-1}}{c}km}.
\end{split}
\end{equation*}
Solving the linear ODE on this interval yields
\begin{equation*}
\begin{split}
 \Lambda_n(x)=\left(-\frac{\lambda}{c}\sum_{i=1}^kP(V_1^{(n-1)}=i)\int_{km}^x\Lambda_n(y-im)e^{-\frac{\lambda q_{n-1}}{c}y}\,dy+C_k\right)e^{\frac{\lambda q_{n-1}}{c}x},
\end{split}
\end{equation*}
for some constant $C_k$. Using the initial condition at $km$, we obtain
\begin{equation*}
\begin{split}
 C_k&=\Lambda_n((k-1)m)e^{-\frac{\lambda q_{n-1}}{c}(k-1)m}-\frac{\lambda}{c}\sum_{i=1}^{k-1}P(V_1^{(n-1)}=i)\int_{(k-1)m}^{km}\Lambda_n(y-im)e^{-\frac{\lambda q_{n-1}}{c}y}\,dy\\
 &=\Lambda_n(km)e^{-\frac{\lambda q_{n-1}}{c}km}.
\end{split}
\end{equation*}
Consequently,
\begin{equation*}
\begin{split}
 \Lambda_n(x)&=\left(\Lambda_n(km)e^{-\frac{\lambda q_{n-1}}{c}km}-\frac{\lambda}{c}\sum_{i=1}^kP(V_1^{(n-1)}=i)\int_{km}^x\Lambda_n(y-im)e^{-\frac{\lambda q_{n-1}}{c}y}\,dy\right)e^{\frac{\lambda q_{n-1}}{c}x}\\
 &=\Lambda_{n,k}(x).
\end{split}
\end{equation*}
This completes the induction and hence the proof.
\end{proof}

\begin{remark}
Observe that when $\xi=m$,  $F(y)$ is a degenerate distribution, then we have
\begin{equation*}
\begin{split}
 G_n(y)&=\sum_{k=1}^{\infty}\frac{P(V_1^{(n-1)}=k)}{q_{n-1}}F^{*k}(y)=\sum_{k=1}^{\infty}\frac{P(V_1^{(n-1)}=k)}{q_{n-1}}\cdot1_{\{y\ge km\}}\\
 &=\sum_{k=1}^{\lfloor y/m\rfloor}\frac{P(V_1^{(n-1)}=k)}{q_{n-1}},
\end{split}
\end{equation*}
which implies that $G_n(y)$ only jumps at points $km$ with hight $p_k=\frac{P(V_1^{(n-1)}=k)}{q_{n-1}}$, if we plug this back into (\ref{deritive_difference}), we can get
\begin{equation*}
\begin{split}
 \Lambda_n^{-}(x_k)-\Lambda_n^{+}(x_k)=\frac{\lambda P(V_1^{(n-1)}=k)}{c}\Lambda_n(0),\quad k\ge1.
\end{split}
\end{equation*}
\end{remark}

\begin{example}
For $n=2$, we have
\begin{equation*}
\begin{split}
 c\Lambda_2'(x)=\lambda q_{1}\Lambda_2(x)-\lambda\sum_{k=1}^{\lfloor x/m\rfloor}P(N_1=k)\Lambda_2(x-km),\quad x\neq km.
\end{split}
\end{equation*}
Choosing $\lambda=m=1$ and $c=2$ yields
\begin{equation*}
\begin{split}
 \Lambda_2'(x)=\frac{1-e^{-1}}{2}\Lambda_2(x)-\frac{e^{-1}}{2}\sum_{k=1}^{\lfloor x\rfloor}\frac{1}{k!}\Lambda_2(x-k),\quad x\neq k,
\end{split}
\end{equation*}
with the initial condition $\Lambda_2(0)=\frac{1}{2}$. Then, according to Theorem \ref{th12_discrete}, the survival probability is
\begin{equation*}
  \Lambda_2(x)=\sum_{k=0}^{\infty}\Lambda_{2,k}(x),
\end{equation*}
where 
\begin{equation*}
\begin{split}
 \Lambda_{2,k}(x)=\left(\Lambda_2(k)e^{-\frac{(1-e^{-1})}{2}k}-\frac{e^{-1}}{2}\sum_{i=1}^k\frac{1}{i!}\int_{k}^x\Lambda_2(y-i)e^{-\frac{1-e^{-1}}{2}y}dy\right)e^{\frac{1-e^{-1}}{2}x}\cdot1_{\{x\in[k,k+1)\}}
\end{split}
\end{equation*}
for $k\ge0$. The corresponding survival probability $\Lambda_2(x)$ is plotted in Figure \ref{fig:survival}. From the figure, we observe that the survival probability starts at $0.5$ and increases towards $1$ as $x$ grows. In particular, at $x=10$, we obtain $\Lambda_2(10)=0.9979$. The differences of the derivatives at the positive integer points are
\begin{equation*}
\begin{split}
 \Lambda_2^{-}(k)-\Lambda_2^{+}(k)=\frac{e^{-1}}{4k!},\quad k\ge1.
\end{split}
\end{equation*}

\begin{figure}[htbp]  
    \centering
    \includegraphics[width=0.6\textwidth]{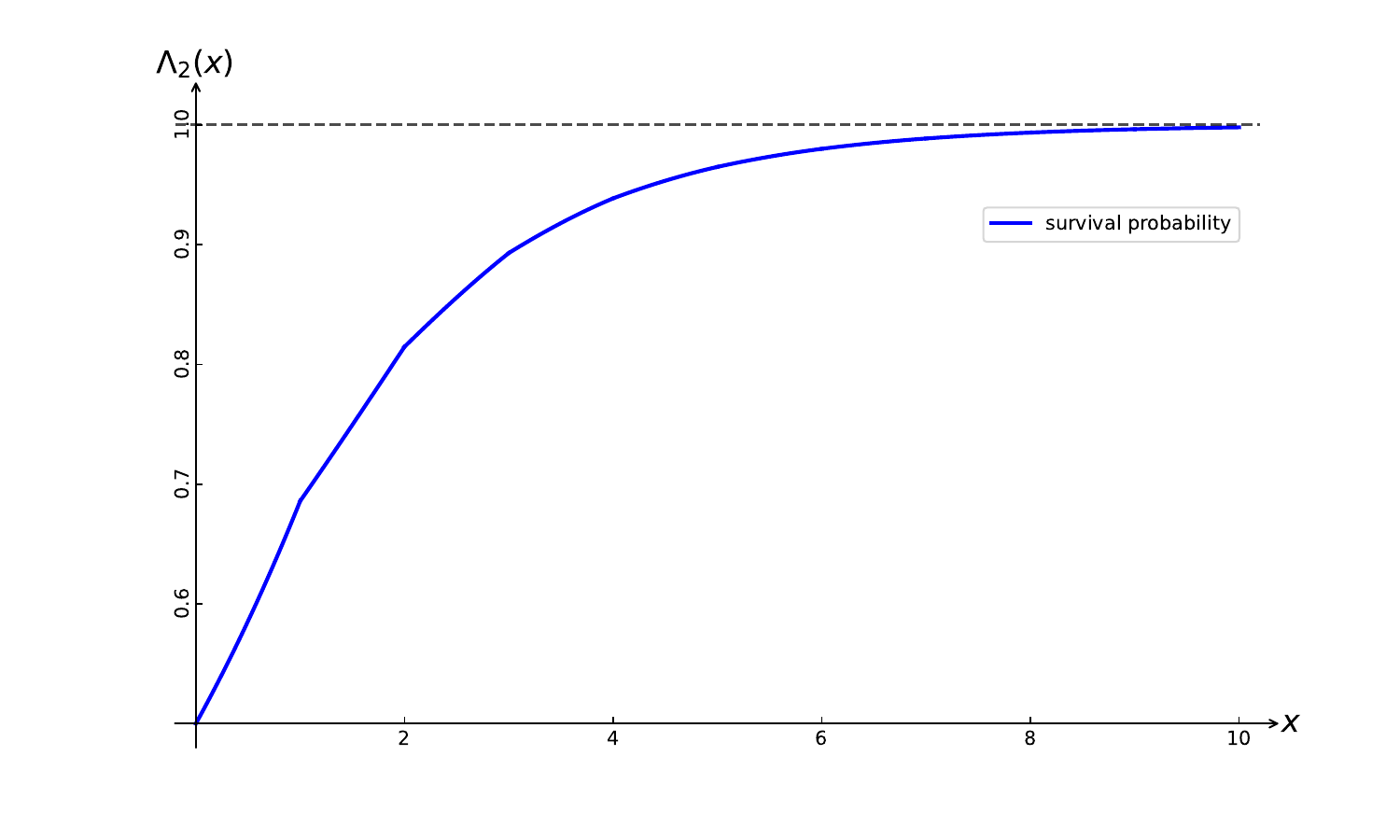} 
    \caption{Survival probability}
    \label{fig:survival}
\end{figure}

\end{example}

\newpage

\section{Laplace transforms of ruin probabilities and scale functions}\label{S:5}

Building on the integro-differential equations derived in Section \ref{S:4}, we now develop the Laplace-transform method for the survival and ruin probabilities. This approach offers a practical route to computing ruin probabilities without directly solving the governing Volterra equation, while at the same time revealing the recursive structure induced by the MIPP hierarchy. We derive a closed-form expression for the Laplace transform of the survival probability and obtain a recursion linking consecutive iteration levels, which yields an efficient computational scheme. We also establish a direct relationship between the survival probability and the scale function of the surplus process, a duality that is central to fluctuation theory and to the asymptotic analysis in the next section. Our treatment follows standard Laplace-transform methods in risk theory; see \cite{Albrecher2010}, \cite{2016Yuliya}, with background on Volterra equations in \cite{Grandell1991}, scale functions in \cite{Kyprianou2014}, and the MIPP construction in \cite{Hu-2026}.

 Recall that $f(x)=e^{-1+x}$ and $g(x)=e^{-\lambda+\lambda x}$ and  the $n-$ folded compositions of $f$ and $g$ are denoted by $f_n^{\circ}$ and $g_n^{\circ}$ respectively. We define the Laplace transforms of $F(y)$ and $\Lambda_n(x)$ by
\begin{equation*}
\begin{split}
 \mathcal{L}_F(\theta)&=\int_0^{+\infty} e^{-\theta y}\,dF(y),\quad \theta\ge0,\\
 \mathcal{L}_{\Lambda_n}(\theta)&=\int_0^{+\infty} e^{-\theta x}\Lambda_n(x)\,dx,\quad \theta>0,
\end{split}
\end{equation*}
respectively. We begin by establishing the following auxiliary lemma.

\begin{lemma}\label{le51}
The Laplace transform of $\kappa_n=:\sum_{i=0}^{V_t^{(n)}}\xi_i$ is given by
\begin{equation}\label{laplace_n}
\begin{split}
 \mathcal{L}_{\kappa_n}(\theta)=E[\exp\{-\theta\sum_{i=0}^{V_t^{(n)}}\xi_i\}]=(g_n^{\circ}(E[e^{-\theta\xi}]))^t.
\end{split}
\end{equation}
\end{lemma}

\begin{proof} We use induction. When $n=1$, we have $V_t^{(n)}=N_t$, therefore 
\begin{equation*}
\begin{split}
 \mathcal{L}_{\kappa_1}(\theta)&=E[\exp\{-\theta\sum_{i=0}^{N_t}\xi_i\}]=\sum_{k=0}^{\infty}P(N_t=k)(E[\exp\{-\theta\xi\}])^k\\
 &=\sum_{k=0}^{\infty}\frac{(\lambda t)^k}{k!}e^{-\lambda t}(E[\exp\{-\theta\xi\}])^k=e^{-\lambda t+\lambda tE[e^{-\theta\xi}]}\\
 &=(g(E[e^{-\theta\xi}]))^t,
\end{split}
\end{equation*}
which satisfies (\ref{laplace_n}). Now, suppose (\ref{laplace_n}) holds for $n$. We need to prove it holds for $n+1$. Observe
\begin{equation*}
\begin{split}
 \sum_{i=0}^{V^{(n+1)}_t}\xi_i\overset{d}{=}\sum_{k=0}^{N_t}(\sum_{i=0}^{V_1^{(n)}}\xi_i),
\end{split}
\end{equation*}
Therefore
\begin{equation*}
\begin{split}
 \mathcal{L}_{\kappa_{n+1}}(\theta)&=E[\exp\{-\theta\sum_{i=0}^{V^{(n+1)}_t}\xi_i\}]=\sum_{k=0}^{\infty}P(N_t=k)(E[\exp\{-\theta\sum_{i=0}^{V_1^{(n)}}\xi_i\}])^k\\
 &=\sum_{k=0}^{\infty}\frac{(\lambda t)^k}{k!}e^{-\lambda t}(g_n^{\circ}(E[e^{-\theta\xi}]))^k=e^{-\lambda t+\lambda t g_n^{\circ}(E[e^{-\theta\xi}])}\\
 &=(g_{n+1}^{\circ}(E[e^{-\theta\xi}]))^t,
\end{split}
\end{equation*}
which completes the proof.
\end{proof}

In the following theorem, we determine the Laplace transform of the survival function.

\begin{theorem}
Suppose $c>\lambda^nm$, then
\begin{equation}\label{laplace_f}
\begin{split}
 \mathcal{L}_{\Lambda_n}(\theta)=\frac{c-\lambda^nm}{c\theta-\lambda(1-g_{n-1}^{\circ}(E[e^{-\theta\xi}]))},\quad \theta>0.
\end{split}
\end{equation}
\end{theorem}

\begin{proof}
Multiplying the first term of (\ref{derivatives}) by $e^{-\theta x}$ and integrating over $[0,\infty)$, we get
\begin{equation}\label{3}
\begin{split}
 c\int_0^{+\infty}\Lambda_n^+(x)e^{-\theta x}dx=&\lambda q_{n-1}\int_0^{+\infty}\Lambda_n(x)e^{-\theta x}dx\\
 -&\lambda\sum_{k=1}^{\infty}P(V_1^{(n-1)}=k)\int_0^{+\infty}\int_0^x\Lambda_n(x-y)dF^{*k}(y)e^{-\theta x}dx.
\end{split}
\end{equation}
We have
\begin{equation*}
\begin{split}
 \int_0^{+\infty}\Lambda_n^+(x)e^{-\theta x}dx&=\int_0^{+\infty}e^{-\theta x}d(\Lambda_n(x))=e^{-\theta x}(\Lambda_n(x))|_0^{+\infty}-\int_0^{+\infty}\Lambda_n(x)de^{-\theta x}\\
 &=-\Lambda_n(0)+\theta\mathcal{L}_{\Lambda_n}(\theta)
\end{split}
\end{equation*}
and
\begin{equation*}
\begin{split}
 \int_0^{+\infty}\int_0^x\Lambda_n(x-y)dF^{*k}(y)e^{-\theta x}dx&=\int_0^{+\infty}\int_y^{+\infty}\Lambda_n(x-y)e^{-\theta x}dxdF^{*k}(y)\\
 &=\int_0^{+\infty}\int_0^{+\infty}\Lambda_n(x)e^{-\theta(x+y)}dxdF^{*k}(y)\\
 &=\int_0^{+\infty}\Lambda_n(x)e^{-\theta x}dx\int_0^{+\infty}e^{-\theta y}dF^{*k}(y)\\
 &=\mathcal{L}_{\Lambda_n}(\theta)\mathcal{L}_{F^{*k}}(\theta)=\mathcal{L}_{\Lambda_n}(\theta)(\mathcal{L}_{F}(\theta))^k.
\end{split}
\end{equation*}
By plugging them into (\ref{3}) and noting that $\Lambda_n(0)=1-\frac{\lambda^nm}{c}$ we obtain
\begin{equation*}
\begin{split}
 -c(1-\frac{\lambda^nm}{c})+c\theta\mathcal{L}_{\Lambda_n}(\theta)=\lambda q_{n-1}\mathcal{L}_{\Lambda_n}(\theta)-\lambda\sum_{k=1}^{\infty}P(V_1^{(n-1)}=k)\mathcal{L}_{\Lambda_n}(\theta)(\mathcal{L}_F(\theta))^k.
\end{split}
\end{equation*}
From this we obtain
\begin{equation*}
\begin{split}
 \mathcal{L}_{\Lambda_n}(\theta)&=\frac{c-\lambda^nm}{c\theta-\lambda(q_{n-1}-\sum_{k=1}^{\infty}P(V_1^{(n-1)}=k)(\mathcal{L}_F(\theta))^k)}\\
 &=\frac{c-\lambda^nm}{c\theta-\lambda(1-E[\exp\{-\theta\sum_{i=0}^{V_1^{(n-1)}}\xi_i\}])}\\
 &=\frac{c-\lambda^nm}{c\theta-\lambda(1-g_{n-1}^{\circ}(E[e^{-\theta\xi}]))}.
\end{split}
\end{equation*}
 
\end{proof}

\begin{remark} The Laplace transform formula in (\ref{laplace_f}) reveals a crucial structural feature of the MIPP-driven ruin model: the denominator depends on the composition $g_{n-1}^{\circ}(\mathcal{L}_F(\theta))$, which encapsulates the hierarchical claim arrival mechanism through the nested function $g$. This recursive dependence is the analytical signature of the multiple subordination construction. In the limiting case $n=1$, the MIPP reduces to the classical Poisson process, and (\ref{laplace_f}) simplifies to the well-known expression
\[
\mathcal{L}_{\Lambda_1}(\theta)=\frac{c-\lambda m}{c\theta-\lambda (1-\mathcal{L}_F(\theta))},
\]
which is consistent with the standard Cramér–Lundberg theory (see, e.g., \cite{asmussen2010}, Eq. (2.13)). For $n\geq 2$, he presence of the composition $g_{n-1}^{\circ}$ introduces additional poles and branch points in the complex plane, which in turn reflect the increased complexity of the clustered arrival structure. Moreover, the numerator $c-\lambda^nm$  confirms the critical role of the iteration depth $n$: as 
$n$ increases, the effective mean claim intensity is amplified by a factor of $\lambda^{n-1}$, a phenomenon already observed in Lemma \ref{le1}. This transform also paves the way for numerical inversion techniques—such as the Euler or Gaver–Stehfest algorithms—which can be employed when explicit inversion is infeasible. Finally, we note that the denominator is strictly positive for $\theta>0$ under the net profit condition $c>\lambda^nm$, ensuring that $\mathcal{L}_{\Lambda_n}(\theta)$ is well-defined and that the survival probability is integrable on $\R_+$, as expected.
\end{remark}

Now let $W_n(x)$ denote the scale function corresponding to the process $X^{(n)}_t(x)$, in the sense that
\begin{equation*}
\begin{split}
 \int_0^{+\infty} e^{-\theta x} W_n(x)\,dx = \frac{1}{\phi_n(\theta)},
\end{split}
\end{equation*}
where $\phi_n(\theta)$ is the Laplace exponent of $X^{(n)}_t(0)$, that is, $E\big[e^{\theta X^{(n)}_t(0)}\big] = e^{\phi_n(\theta)t}$. Based on (\ref{laplace_f}), we can derive the following recursive formula.

\begin{corollary}\label{cor53} i)
The Laplace transform of the survival probability satisfies the incursive relation
\begin{equation*}
\begin{split}
 \mathcal{L}_{\Lambda_{n+1}}(\theta)=\frac{c-\lambda^{n+1}m}{c\theta-\lambda\left(1-\exp\left\{\frac{c-\lambda^nm}{\mathcal{L}_{\Lambda_n}(\theta)}-c\theta\right\}\right)}.
\end{split}
\end{equation*}
ii) The scale function associated with $X_t^{(n)}(x)$ satisfies
\begin{equation*}
\begin{split}
 \mathcal{L}_{W_n}(\theta) = \frac{\mathcal{L}_{\Lambda_n}(\theta)}{c - \lambda^n m}.
\end{split}
\end{equation*}
iii) The recursive representation of the Laplace transform of the scale function is
\begin{equation}\label{recursive}
\begin{split}
 \mathcal{L}_{W_{n+1}}(\theta)=\frac{1}{c\theta-\lambda\left(1-\exp\left\{\frac{1}{\mathcal{L}_{W_n}(\theta)}-c\theta\right\}\right)}.
\end{split}
\end{equation}
\end{corollary}

\begin{proof} i) 
Starting from (\ref{laplace_f}), we obtain
\begin{equation}\label{mgf1}
\begin{split}
 g_{n-1}^{\circ}(E[e^{-\theta\xi}])=\frac{c-\lambda^nm-(c\theta-\lambda)\mathcal{L}_{\Lambda_n}(\theta)}{\lambda\mathcal{L}_{\Lambda_n}(\theta)}.
\end{split}
\end{equation}
Plugging (\ref{mgf1}) into the representation of $\mathcal{L}_{\Lambda_{n+1}}(\theta)$ gives
\begin{equation*}
\begin{split}
 \mathcal{L}_{\Lambda_{n+1}}(\theta)
   &=\frac{c-\lambda^{n+1}m}{c\theta-\lambda\bigl(1-g_{n}^{\circ}(E[e^{-\theta\xi}])\bigr)}\\
   &=\frac{c-\lambda^{n+1}m}{c\theta-\lambda\bigl(1-g(g_{n-1}^{\circ}(E[e^{-\theta\xi}]))\bigr)}\\
   &=\frac{c-\lambda^{n+1}m}{c\theta-\lambda\left(1-\exp\left\{\frac{c-\lambda^nm}{\mathcal{L}_{\Lambda_n}(\theta)}-c\theta\right\}\right)}.
\end{split}
\end{equation*}

ii) By Lemma \ref{le51}, we have
\begin{equation*}
\begin{split}
 E\big[e^{\theta X^{(n)}_t(0)}\big]
   = e^{c\theta t} E\big[e^{-\theta\sum_{i=0}^{V_t^{(n)}} \xi_i}\big]
   = e^{c\theta t} \big(g_n^{\circ}(E[e^{-\theta \xi}])\big)^t,
\end{split}
\end{equation*}
and therefore
\begin{equation*}
\begin{split}
 \mathcal{L}_{W_n}(\theta)
   = \int_0^{+\infty} e^{-\theta x} W_n(x)\,dx
   = \frac{1}{c\theta + \ln\big(g_n^{\circ}(E[e^{-\theta \xi}])\big)}
   = \frac{1}{c\theta - \lambda\big(1 - g_{n-1}^{\circ}(E[e^{-\theta \xi}])\big)}.
\end{split}
\end{equation*}
Comparing this identity with (\ref{laplace_f}) leads to the asserted relation.

iii) This can easily be obtained combining i) and ii).
\end{proof}

\begin{remark} The identity in Corollary \ref{cor53}, establishes a direct and remarkably simple relationship between the scale function and the survival probability in the MIPP-driven risk model. This result is not merely a formal curiosity; it has profound practical implications. First, it implies that the scale function 
$W_n(x),$  which is traditionally defined through the Laplace exponent of the surplus process and often requires sophisticated fluctuation-theoretic arguments to compute, can be obtained immediately once the survival probability is known. Conversely, the survival probability—and hence the ruin probability via $\Pi_n(x)=1-(c-\lambda^nm)W_n(x)$—can be recovered directly from the scale function, confirming the well-known duality between these two objects in Lévy process theory (see, e.g., \cite{Kyprianou2014}, Chapter 8). Second, the constant factor $c-\lambda^nm$ in the denominator represents the positive safety loading of the insurer under the net profit condition, and its appearance underscores the fact that the scale function is inversely proportional to the drift of the surplus process. Third, this identity offers a computationally efficient route for evaluating ruin probabilities: one may either invert the Laplace transform of $\Lambda_n(x)$ via (\ref{laplace_f}) or invert that of $W_n(x)$ via Corollary \ref{cor53}, whichever is more convenient for a given numerical scheme. Finally, we emphasize that this relationship is not an artifact of the MIPP structure but rather a general property of spectrally positive Lévy processes with a positive drift; its validity in our setting serves as a consistency check and reinforces the fact that the MIPP-driven surplus process falls squarely within the broader class of Lévy processes with well-defined scale functions. This connection will be exploited in the next section to derive the Cramér–Lundberg approximation, where the asymptotic behavior of $W_n(x)$ as $n\rightarrow \infty$ directly determines the ruin probability's decay rate.

\end{remark}
\begin{remark} Define $h_n(\theta)=\frac{1}{\mathcal{L}_{W_n}(\theta)}-c\theta+\lambda=\phi_n(\theta)-c\theta+\lambda$. It is easy to see that  (\ref{recursive}) implies the following recursive relation
\begin{equation*}
\begin{split}
 h_{n+1}(\theta)=\lambda e^{-\lambda}e^{h_n(\theta)},
\end{split}
\end{equation*}
with initial condition
\begin{equation*}
\begin{split}
 h_1(\theta)=\phi_1(\theta)-c\theta+\lambda=\lambda E[e^{-\theta\xi}].
\end{split}
\end{equation*}
Using this recursive relation one easily obtains $h_n(\theta)=\lambda g_{n-1}^{\circ}(E[e^{-\theta\xi}])$ for all $n\ge2$. This is an equivalent form for the relation for  $\mathcal{L}_{W_n}(\theta)$ in the proof of the Corollary \ref{cor53}.
\end{remark}

\section{Cram\'er-Lundberg approximation}\label{S:6}

We now  study the asymptotic behaviour of the ruin probability as the initial reserve $x$ grows large. Although the transform-based techniques of Section \ref{S:5} provide exact numerical values, they are often not directly interpretable in practical risk management or regulatory settings. We therefore develop a Cramér–Lundberg-type approximation for the MIPP-driven risk model, which yields a simple yet robust exponential asymptotic expression for $\Pi_n(x)$. The approximation is built on a unique adjustment coefficient $\hat{R}_n>0$, or Lundberg exponent, defined as the positive solution of the generalized Cramér–Lundberg equation
\[
\lambda (g^{\circ}_{n-1}(E[e^{R\xi}]-1))=cR,
\]
which captures the hierarchical structure of the claim arrivals. In this section we derive an explicit formula for the asymptotic constant $C_n$, prove the asymptotic relation
\[
\Pi_n(x)\sim C_ne^{-\hat{R}_nx}\quad (x\rightarrow \infty),
\]
and establish the corresponding Lundberg inequality $\Pi_n(x)\le e^{-\hat{R}_nx}$ via a martingale argument. A numerical comparison with exact Laplace inversion confirms the accuracy of the approximation. Related asymptotic results for classical and generalized risk models can be found in \cite{Albrecher2010}, \cite{Baltrunas-1999}, and \cite{Albrecher2020}, with further contributions in \cite{Barrow-1936}, \cite{Constantinescu2018}, \cite{Palmowski2025}; for the MIPP construction and scale-function background, see \cite{Hu-2026} and \cite{Kyprianou2014}, while general Lundberg-type bounds are discussed in \cite{2016Yuliya}.

Before we derive a Cram\'er-Lundberg approximation in our setting, we first prove few Lemmas. Recall that $g(z)=e^{-\lambda+\lambda z}$ and $g_{n-1}^{\circ}$ denotes $(n-1)-$folded composition of $g(z)$.

\begin{lemma}\label{lem6.1} For each fixed $n\geq 2$, assume the net profit condition $c>\lambda^nm$ holds in our model (\ref{one1}). Also assume $\mathcal{L}_{\xi}(\theta)$ is finite. Then the following equation 
\begin{equation}\label{Lundberg_equation-0}
\lambda\bigl(g_{n-1}^{\circ}(E[e^{R\xi}]) - 1\bigr) = cR,
\end{equation}
has unique positive solution, which we denote by $\hat{R}_n$.    
\end{lemma}
\begin{proof}Define
\[
\Phi_n(R)=\lambda\bigl(g_{n-1}^{\circ}(E[e^{R\xi}]) - 1\bigr)-cR.
\]
We want to prove that there is a unique $R>0$ such that $\Phi_n(R)=0$.

First note that $g(1)=e^{-\lambda+\lambda}=1$, and by induction this implies $g_{n-1}^{\circ}(1)=1$ as well. Consequently, $\Phi_n(0)=0$, so $R=0$ is always a root. Next,
\[
\Phi'_n(0)=\lambda \frac{d}{dR}g_{n-1}^{\circ}(Ee^{R\xi})\bigg|_{R=0}-c.
\]
Applying the chain rule and using $g_{n-1}^{\circ}(1)=1$ yields
\[
\frac{d}{dR}g_{n-1}^{\circ}(Ee^{R\xi})\bigg|_{R=0}
=(g_{n-1}^{\circ})'(1)\,E\xi.
\]
We now compute $(g_{n-1}^{\circ})'(1)$. For $n=2$, $(g_1^{\circ})'(1)=g'(1)=\lambda$. By induction this generalizes to $(g_{n-1}^{\circ})'(1)=\lambda^{\,n-1}$. Hence
\[
\Phi'_n(0)=\lambda^n m - c,
\]
where $m=E\xi$. Under the net profit condition ($\lambda^n m<c$) we have $\Phi'_n(0)<0$. Thus, starting from $\Phi_n(0)=0$, the function $\Phi_n(R)$ initially decreases for small positive $R$.

We now consider the behavior as $R\to\infty$. Since $\xi$ is strictly positive with positive probability, $E[e^{R\xi}]\to\infty$ as $R\to\infty$. As $z\to\infty$,
\[
g(z)=e^{-\lambda+\lambda z}\to\infty
\]
super-exponentially, and therefore $g_{n-1}^{\circ}(E[e^{R\xi}])\to\infty$ as $R\to\infty$ at a rate faster than any linear function (for $n\ge2$). Consequently,
\[
\frac{cR}{\lambda\, g_{n-1}^{\circ}(E[e^{R\xi}])}\to 0 \quad \text{as } R\to\infty,
\]
so
\[
\lim_{R\to\infty}\Phi_n(R)=\infty.
\]

Next we verify strict convexity of $\Phi_n(R)$ in $R$. The map $R\mapsto E[e^{R\xi}]$ is strictly convex, and $g$ is increasing and strictly convex. Composition of a strictly convex function with an increasing strictly convex function preserves strict convexity, so $R\mapsto g_{n-1}^{\circ}(E[e^{R\xi}])$ is strictly convex. Since $\Phi_n(R)$ is this term multiplied by $\lambda$ minus a linear term $cR$, $\Phi_n$ is strictly convex as well.

We therefore have a strictly convex function $\Phi_n$ with $\Phi_n(0)=0$, $\Phi'_n(0)<0$, and $\lim_{R\to\infty}\Phi_n(R)=\infty$. Such a function crosses the horizontal axis exactly once for $R>0$. Hence there exists a unique positive solution $R>0$ to $\Phi_n(R)=0$.
\end{proof}

We start by proving the following lemma.

\begin{lemma}\label{lem_exp}
For the MIPP $V_t^{(n)}$, the following identity holds:
\begin{equation}\label{lap_mix}
\begin{split}
 E\bigl[V_t^{(n)}z^{V_t^{(n)}}\bigr]
 =\lambda^n t\bigl(g_n^{\circ}(z)\bigr)^t\prod_{i=0}^{n-1}g_i^{\circ}(z),
\end{split}
\end{equation}
where $g_0^{\circ}(z)=z$ and $g_1^{\circ}(z)=g(z)=e^{-\lambda+\lambda z}$.
\end{lemma}

\begin{proof}
The proof is by induction on $n$. For $n=1$ we obtain
\begin{equation*}
\begin{split}
 E[N_tz^{N_t}]
 &=\sum_{k=0}^{+\infty}P(N_t=k)\,kz^k
 =e^{-\lambda t}\sum_{k=0}^{+\infty}\frac{(\lambda t)^k}{k!}kz^k\\
 &=\lambda t z\, e^{-\lambda t}\sum_{k=0}^{+\infty}\frac{(\lambda t z)^k}{k!}
 =\lambda t z\, e^{-\lambda t+\lambda t z}
 =\lambda t z\,(g(z))^t,
\end{split}
\end{equation*}
which coincides with (\ref{lap_mix}) in the case $n=1$. Now assume the identity holds for some $n$. We then treat the case $n+1$:
\begin{equation*}
\begin{split}
 E\bigl[V_t^{(n+1)}z^{V_t^{(n+1)}}\bigr]
 &=\sum_{k=0}^{+\infty}P(N_t^1=k)\,E\bigl[V_k^{(n)}z^{V_k^{(n)}}\bigr]\\
 &=\sum_{k=0}^{+\infty}\frac{(\lambda t)^k}{k!}e^{-\lambda t}\cdot
   \lambda^n k\bigl(g_n^{\circ}(z)\bigr)^k\prod_{i=0}^{n-1}g_i^{\circ}(z)\\
 &=\lambda^n\prod_{i=0}^{n-1}g_i^{\circ}(z)\,e^{-\lambda t}\lambda t g_n^{\circ}(z)
   \sum_{k=0}^{+\infty}\frac{\bigl(\lambda t g_n^{\circ}(z)\bigr)^k}{k!}\\
 &=\lambda^{n+1}t\prod_{i=0}^n g_i^{\circ}(z)\,
   e^{-\lambda t+\lambda t g_n^{\circ}(z)}\\
 &=\lambda^{n+1}t\bigl(g_{n+1}^{\circ}(z)\bigr)^t\prod_{i=0}^n g_i^{\circ}(z),
\end{split}
\end{equation*}
which establishes the statement for $n+1$ and therefore completes the induction.
\end{proof}

Then, we have the Cram\'er-Lundberg approximation for the model (\ref{one1}) as follows:
\begin{theorem}\label{th63}
Assume that $\hat{R}_n>0$ is the solution to
\begin{equation}\label{Lundberg_equation}
\begin{split}
 \lambda\bigl(g_{n-1}^{\circ}(E[e^{R\xi}]) - 1\bigr) = cR,
\end{split}
\end{equation}
then the ruin probability admits the asymptotic representation
\begin{equation*}
\begin{split}
 \Pi_n(x)\sim C_ne^{-\hat{R}_nx}\quad\text{as} \quad x\to+\infty,
\end{split}
\end{equation*}
where
\begin{equation}\label{C}
\begin{split}
 C_n=\frac{c-\lambda^nm}{\lambda^nE[\xi e^{\hat{R}_n\xi}]\prod_{i=1}^{n-1}g_i^{\circ}(E[e^{\hat{R}_n\xi}])-c}.
\end{split}
\end{equation}

\end{theorem}
\begin{proof}
We first denote by $W_n$ the jump sizes of $V^{(n-1)}_{J_1^{(n)}}$, i.e.
\begin{equation*}
\begin{split}
 P(W_n\le y)=\int_0^y dG_n(x)=\sum_{k=1}^{\infty}\frac{P(V_1^{(n-1)}=k)}{q_{n-1}}\int_0^y dF^{*k}(x),
\end{split}
\end{equation*}
which implies
\begin{equation}\label{w_n}
\begin{split}
 E[e^{\theta W_n}]&=\sum_{k=1}^{\infty}\frac{P(V_1^{(n-1)}=k)}{q_{n-1}}\int_0^{\infty}e^{\theta x}dF^{*k}(x)\\
 &=\frac{1}{q_{n-1}}\left(\sum_{k=0}^{\infty}P(V_1^{(n-1)}=k)(E[e^{\theta\xi}])^k-P(V_1^{(n-1)}=0)\right)\\
 &=\frac{1}{q_{n-1}}\left(g_{n-1}^{\circ}(E[e^{\theta\xi}])-(1-q_{n-1})\right)\\
 &=\frac{1}{q_{n-1}}\bigl(g_{n-1}^{\circ}(E[e^{\theta\xi}])-1\bigr)+1.
\end{split}
\end{equation}

Hence, the Cramér–Lundberg equation can be rewritten as
\begin{equation*}
\begin{split}
 \lambda q_{n-1}(E[e^{RW_n}]-1)=cR.
\end{split}
\end{equation*}
Substituting (\ref{w_n}) into this equation yields (\ref{Lundberg_equation}), and the root $\hat{R}_n$ is the usual adjustment coefficient (or Lundberg exponent).

Next, by \cite{asmussen2010}, \cite{2016Yuliya}, we have
\begin{equation}\label{c_n}
\begin{split}
 C_n=\frac{c-\lambda q_{n-1}m_{W_n}}{\lambda q_{n-1}E[W_ne^{\hat{R}_nW_n}]-c},
\end{split}
\end{equation}
where
\begin{equation}\label{mu_n}
\begin{split}
 m_{W_n}&=\int_0^{+\infty}x\,dG_n(x)=\sum_{k=1}^{\infty}\frac{P(V_1^{(n-1)}=k)}{q_{n-1}}\int_0^{+\infty}x\,dF^{*k}(x)
 =\sum_{k=1}^{\infty}\frac{P(V_1^{(n-1)}=k)km}{q_{n-1}}\\
 &=\frac{m}{q_{n-1}}E[V_1^{(n-1)}]=\frac{\lambda^{n-1}m}{q_{n-1}},
\end{split}
\end{equation}
and, using Lemma \ref{lem_exp},
\begin{equation}\label{mix_wn}
\begin{split}
 E[W_ne^{\hat{R}_nW_n}]&=\sum_{k=1}^{+\infty}\frac{P(V_1^{(n-1)}=k)}{q_{n-1}}E\left[\sum_{i=1}^k\xi_ie^{\hat{R}_n\sum_{i=1}^k\xi_i}\right]\\
 &=\sum_{k=1}^{+\infty}\frac{P(V_1^{(n-1)}=k)}{q_{n-1}}kE[\xi e^{\hat{R}_n\xi}](E[e^{\hat{R}_n\xi}])^{k-1}\\
 &=\frac{1}{q_{n-1}}\frac{E[\xi e^{\hat{R}_n\xi}]}{E[e^{\hat{R}_n\xi}]}\sum_{k=0}^{+\infty}P(V_1^{n-1}=k)k(E[e^{\hat{R}_n\xi}])^k\\
 &=\frac{1}{q_{n-1}}\frac{E[\xi e^{\hat{R}_n\xi}]}{E[e^{\hat{R}_n\xi}]}\,E\bigl[V_1^{(n-1)}(E[e^{\hat{R}_n\xi}])^{V_1^{(n-1)}}\bigr]\\
 &=\frac{1}{q_{n-1}}\frac{E[\xi e^{\hat{R}_n\xi}]}{E[e^{\hat{R}_n\xi}]}\cdot\lambda^{n-1}\prod_{i=0}^{n-1}g_i^{\circ}(E[e^{\hat{R}_n\xi}])\\
 &=\frac{\lambda^{n-1}}{q_{n-1}}E[\xi e^{\hat{R}_n\xi}]\prod_{i=1}^{n-1}g_i^{\circ}(E[e^{\hat{R}_n\xi}]).
\end{split}
\end{equation}
Finally, inserting (\ref{mu_n}) and (\ref{mix_wn}) into (\ref{c_n}) yields the expression for $C_n$ in (\ref{C}).
\end{proof}

\begin{example}
When $n=2$, the Cram\'er–Lundberg equation reduces to
\begin{equation*}
\begin{split}
 \lambda e^{-\lambda+\lambda E[e^{R\xi}]}-\lambda =cR
\end{split}
\end{equation*}
and
\begin{equation*}
\begin{split}
 C_2=\frac{c-\lambda^2m}{\lambda^2E[\xi e^{\hat{R}_n\xi}]e^{-\lambda+\lambda E[e^{\hat{R}_n\xi}]}-c}.
\end{split}
\end{equation*}

If, in addition, we assume that $\xi$ is exponentially distributed, the Cram\'er–Lundberg equation becomes
\begin{equation*}
\begin{split}
 \lambda e^{-\lambda+\frac{\lambda}{1-m R}}-\lambda=cR
\end{split}
\end{equation*}
and
\begin{equation*}
\begin{split}
 C_2=\frac{c-\lambda^2m}{\frac{\lambda^2m e^{-\lambda+\frac{\lambda}{1-m\hat{R}_2}}}{(1-m\hat{R}_2)^2}-c}
      =\frac{c-\lambda^2m}{\frac{\lambda m (c\hat{R}_2+\lambda)}{(1-m\hat{R}_2)^2}-c}.
\end{split}
\end{equation*}

Furthermore, taking $m=\lambda=1$ and $c=2$, the Cram\'er–Lundberg equation simplifies to
\begin{equation*}
\begin{split}
 e^{-1+\frac{1}{1-R}}-1=2R,
\end{split}
\end{equation*}
from which we obtain $\hat{R}_2=0.3433$. Substituting this value yields
\begin{equation*}
\begin{split}
 C_2=\frac{1}{\frac{2\hat{R}_2+1}{(1-\hat{R}_2)^2}-2}=0.5235,
\end{split}
\end{equation*}
and thus the ruin probability satisfies
\begin{equation*}
\begin{split}
 \Pi_2(x)\sim 0.5235e^{-0.3433x}\quad\text{as}\quad x\to+\infty.
\end{split}
\end{equation*}

To evaluate the quality of this approximation, we plot a graph comparing the approximation formula with the numerical inverse Laplace transform method, which serves as a benchmark.

\begin{figure}[htbp]  
    \centering
    \includegraphics[width=0.8\textwidth]{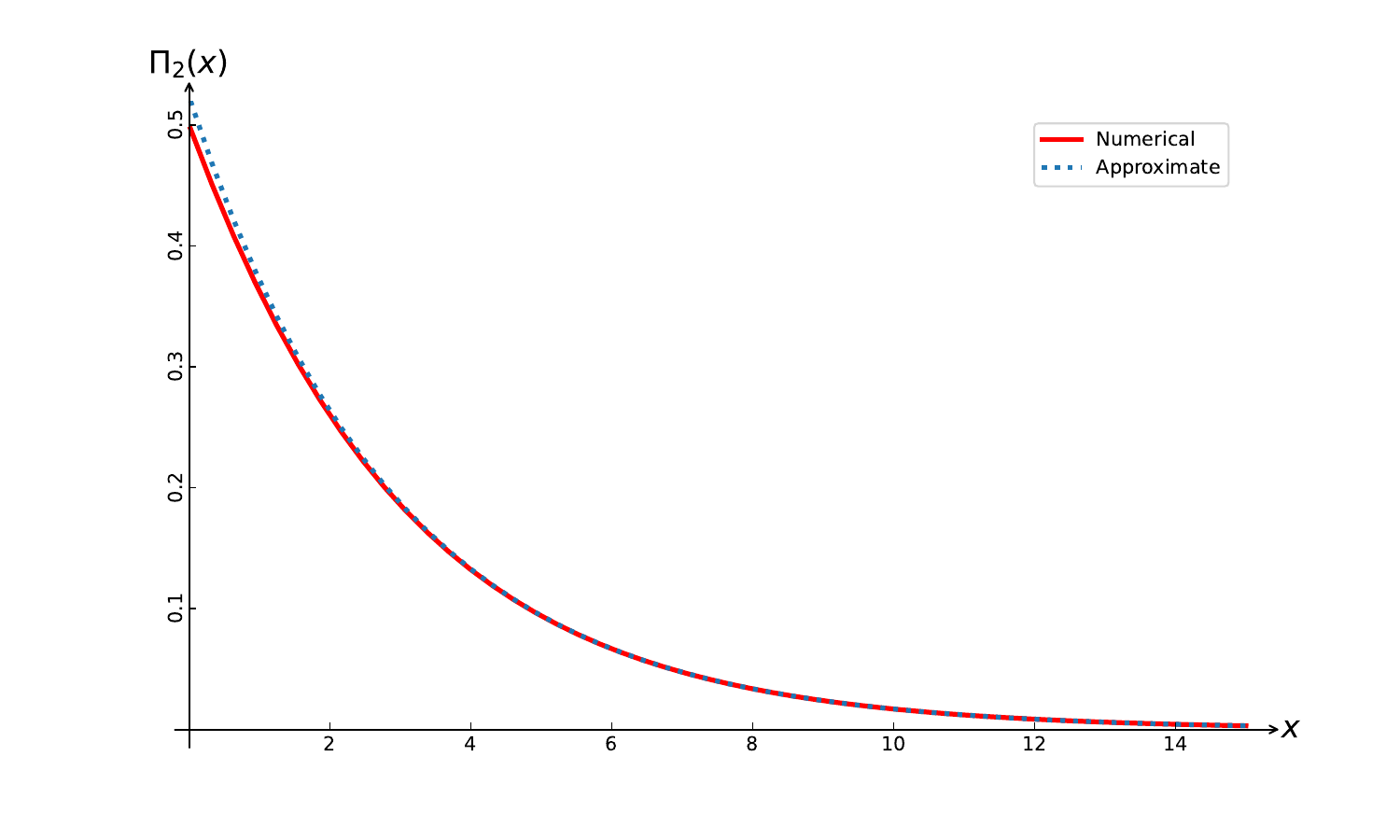} 
    \caption{Comparison of the Cram\'er–Lundberg approximation with numerical method}
    \label{fig:1}
\end{figure}
From Figure \ref{fig:1}, we observe that, as $x\to+\infty$, the approximation becomes indistinguishable from the numerical results.
\end{example}



\newpage

\begin{theorem}
With $\hat{R}_n$ as introduced in Theorem \ref{th63}, the process
\begin{equation*}
\begin{split}
 \exp\left\{-\hat{R}_n\left(ct-\sum_{i=0}^{V_t^{(n)}}\xi_i\right)\right\}
\end{split}
\end{equation*}
forms a martingale. Moreover, the following inequality holds:
\begin{equation}\label{lundberg_inequality}
\begin{split}
 \Pi_n(x)\le e^{-\hat{R}_n x}.
\end{split}
\end{equation}
\end{theorem}

\begin{proof}
First, note that $\exp\{-\hat{R}_n(ct-\sum_{i=0}^{V_t^{(n)}}\xi_i)\}$ is clearly adapted to the filtration $\mathcal{F}_n$. Moreover,
\begin{equation*}
\begin{split}
 E\Big[\exp\Big\{-\hat{R}_n(ct-\sum_{i=0}^{V_t^{(n)}}\xi_i)\Big\}\Big]
 &=e^{-\hat{R}_nct}E\Big[\exp\Big\{\hat{R}_n\sum_{i=0}^{V_t^{(n)}}\xi_i\Big\}\Big]\\
 &=\exp\Big\{t\big(-\hat{R}_nc+\lambda(g_{n-1}^{\circ}(E[e^{\hat{R}_n\xi}])-1)\big)\Big\}\\
 &=1<\infty,
\end{split}
\end{equation*}
which shows that $\exp\{-\hat{R}_n(ct-\sum_{i=0}^{V_t^{(n)}}\xi_i)\}$ is integrable. Furthermore, for any $0<s<t$, we have
\begin{equation*}
\begin{split}
 E\Big[\exp\Big\{-\hat{R}_n(ct-\sum_{i=0}^{V_t^{(n)}}\xi_i)\Big\}\Big|\mathcal{F}_s\Big]
 &=E\Big[\exp\Big\{-\hat{R}_n(ct-\sum_{i=0}^{V_t^{(n)}}\xi_i)\Big\}\Big|\mathcal{F}_s\Big]\\
 &=\exp\Big\{-\hat{R}_n(cs-\sum_{i=0}^{V_s^{(n)}}\xi_i)\Big\} \,
    E\Big[\exp\Big\{-\hat{R}_n\Big((t-s)-\sum_{i=V_s^{(n)}}^{V_t^{(n)}}\xi_i\Big)\Big\}\Big]\\
 &=\exp\Big\{-\hat{R}_n(cs-\sum_{i=0}^{V_s^{(n)}}\xi_i)\Big\} \,
    E\Big[\exp\Big\{-\hat{R}_n\Big((t-s)-\sum_{i=0}^{V_{t-s}^{(n)}}\xi_i\Big)\Big\}\Big]\\
 &=\exp\Big\{-\hat{R}_n(cs-\sum_{i=0}^{V_s^{(n)}}\xi_i)\Big\}.
\end{split}
\end{equation*}
The third equality uses the independent and stationary increments of $V_t^{(n)}$, and the fourth follows from Theorem \ref{th63}.  

To prove the second claim, define the ruin time by $T_n(x)=\inf\{t>0: X_t^{(n)}(x)<0\}$. Then $T_n(x)$ is evidently an $\mathcal{F}_t$-stopping time. Hence, for any $t>0$, the random time $T_n(x)\wedge t$ is a bounded $\mathcal{F}_t$-stopping time. Thus, by the optional sampling theorem,
\begin{equation}\label{martingale}
\begin{split}
 1
 &=E\Big[\exp\Big\{-\hat{R}_n(c\cdot 0-\sum_{i=0}^{V_0^{(n)}}\xi_i)\Big\}\Big]\\
 &=E\Big[\exp\Big\{-\hat{R}_n(c(T_n(x)\wedge t)-\sum_{i=0}^{V_{T_n(x)\wedge t}^{(n)}}\xi_i)\Big\}\Big]\\
 &=P(T_n(x)<t)\,
    E\Big[\exp\Big\{-\hat{R}_n(cT_n(x)-\sum_{i=0}^{V_{T_n(x)}^{(n)}}\xi_i)\Big\}\Big]\\
 &\quad +P(T_n(x)\ge t)\,
    E\Big[\exp\Big\{-\hat{R}_n(ct-\sum_{i=0}^{V_t^{(n)}}\xi_i)\Big\}\Big]\\
 &\ge P(T_n(x)<t)\,
    E\Big[\exp\Big\{-\hat{R}_n(cT_n(x)-\sum_{i=0}^{V_{T_n(x)}^{(n)}}\xi_i)\Big\}\Big]\\
 &\ge P(T_n(x)<t)e^{\hat{R}_nx},
\end{split}
\end{equation}
where the last inequality comes from the fact that when ruin occurs, the surplus process must be strictly less than $-x$. From (\ref{martingale}) we obtain
\begin{equation*}
 P(T_n(x)<t)\le e^{-\hat{R}_nx}.
\end{equation*}
Letting $t\to+\infty$ yields
\begin{equation*}
 \Pi_n(x)\le e^{-\hat{R}_nx}.
\end{equation*}

\end{proof}

The inequality (\ref{lundberg_inequality}) is commonly referred to in the literature as the Lundberg inequality.

\section{Conclusion}\label{S:7}

This paper develops a systematic ruin theory for an insurer’s surplus process when the driving claim-arrival mechanism is a Multiply Iterated Poisson Process (MIPP), thereby generalizing the classical Cramér–Lundberg model to situations where claims occur in hierarchical clusters. The MIPP is constructed by recursively composing independent Poisson processes, yielding a non-negative Lévy subordinator that models multi-stage or cascading claim arrivals while preserving analytical tractability.

We begin by establishing the generalized net profit condition. Ruin is inevitable when the premium rate is at or below a critical level determined by the mean claim amount, the intensity parameter, and the iteration depth. If the premium rate exceeds this level, ruin is no longer certain, and the ruin probability vanishes as the initial surplus increases. In this way, the iteration depth increases the effective claim intensity. Moreover, letting the iteration depth grow without bound reveals a phase transition: ruin becomes almost sure when the intensity parameter is greater than one, converges to a classical constant when the intensity parameter equals one, and tends to zero when the intensity parameter is less than one. This underscores the need to estimate both claim intensity and hierarchical depth accurately in catastrophic risk settings.

For a fixed finite iteration depth, we derive integro-differential equations for the survival probability and show that these can be reformulated as a second-kind Volterra integral equation and, for discrete claim sizes, as delay differential equations. Using a full Laplace-transform approach, we obtain an explicit expression for the Laplace transform of the survival probability together with a recursion that connects successive iteration levels. We further prove a duality between the survival probability and the scale function associated with the surplus process. This relationship provides an efficient way to compute ruin probabilities and situates the model within the fluctuation theory of spectrally negative Lévy processes.

Finally, we construct a Cramér–Lundberg-style approximation featuring an explicit adjustment coefficient and asymptotic constant, and we establish the corresponding Lundberg inequality. Numerical results support the quality of the approximation and demonstrate how iteration depth, claim intensity, and the claim-size distribution interact in determining ruin probabilities.

In summary, the MIPP framework extends the Cramér–Lundberg paradigm in a theoretically natural and practically meaningful direction: claim arrivals are no longer isolated Poisson events but instead arise from an iterated, hierarchical Poisson mechanism. The classical model corresponds to the first level $n=1$, while higher levels represent clustered or cascading losses. The resulting ruin theory preserves the familiar components—net profit condition, integro-differential equation, Laplace transform, scale function, adjustment coefficient, and Lundberg bound—while augmenting each with a recursive dependence on the iteration depth $n$ and the intensity $\lambda$.

These findings carry practical relevance for solvency evaluation, capital planning, and reinsurance structuring in business lines exposed to clustered or cascading losses. A number of avenues for future work remain. It would be valuable to introduce dependence between claim severities and the arrival hierarchy, to enrich the iterated Poisson setup with regime shifts or stochastic intensities, to analyze ruin probabilities over a finite time horizon, and to generalize the approach to multivariate or spatial risk settings. Developing statistical methods to infer Multiply Iterated Poisson Process parameters in the presence of clustered observations is also a promising topic for further study.

\newpage
\printbibliography

\end{document}